\documentclass[11pt]{amsart}

\usepackage{latexsym}
\usepackage{amsmath}
\usepackage{amsthm}
\usepackage{amssymb}
\usepackage{amsfonts}
\usepackage{comment}

\usepackage{mathrsfs}
\usepackage[ocgcolorlinks, linkcolor=blue]{hyperref}

\usepackage{calc}
\newenvironment{tehtratk}%
             {\begin{list}{\arabic{enumi}.}{\usecounter{enumi}%
              \setlength{\labelsep}{0.5em}%
              \settowidth{\labelwidth}{\arabic{enumi}.}%
              \setlength{\leftmargin}{\labelwidth+\labelsep}}}%
             {\end{list}}

             {\begin{list}{$\bullet$}{\usecounter{enumi}%
              \setlength{\labelsep}{0.5em}%
              \settowidth{\labelwidth}{*}%
              \setlength{\leftmargin}{\labelwidth+\labelsep}}}%
             {\end{list}}

\DeclareRobustCommand{\SkipTocEntry}[5]{}

\newcommand{\mR}{\mathbb{R}}                    
\newcommand{\abs}[1]{\lvert #1 \rvert}          
\newcommand{\norm}[1]{\lVert #1 \rVert}         

\newcommand{\ol}[1]{\overline{#1}}

\newcommand{\etilde}{\,\tilde{\rule{0pt}{6pt}}\,}

\newcommand{\im}{\mathrm{Im}}

\newcommand{\mF}{\mathscr{F}}

\newcommand{\p}{\partial}
\newcommand{\closure}[1]{\overline{#1}}
\newcommand{\dbar}{\overline{\partial}}

\newcommand{\mdiv}{\mathrm{div}}

\def \vd{\overset{\tt{v}}{\nabla}}
\def \hd{\overset{\tt{h}}{\nabla}}

\newcommand{\eps}{\varepsilon}

\newtheorem{theorem}{Theorem}[section]

\newtheorem{lemma}[theorem]{Lemma}

\newtheorem{proposition}[theorem]{Proposition}

\newtheorem{question}{Question}[section]

\theoremstyle{definition}
\newtheorem*{definition}{Definition}
\newtheorem{example}[theorem]{Example}

\newtheorem{exercise}{Exercise}[section]
\newtheorem{remark}[theorem]{Remark}

\numberwithin{equation}{section}

\begin{document}

\title{On geometric inverse problems and microlocal analysis}
\author{Mikko Salo}
\address{Department of Mathematics and Statistics, University of Jyv\"askyl\"a}
\email{mikko.j.salo@jyu.fi}

\begin{abstract}
This work gives an expository account of certain applications of microlocal analysis in three geometric inverse problems. We will discuss the geodesic X-ray transform inverse problem, the Gelfand problem for the wave equation on a Riemannian manifold, and the Calder\'on problem for the Laplace equation on a Riemannian manifold.
\end{abstract}

\maketitle

\tableofcontents



\section{Introduction}

In these notes we will describe certain applications of microlocal analysis in geometric inverse problems. We will discuss the following important problems:
\begin{enumerate}
\item[1.]
determining a function from its integrals over maximal geodesics (geodesic X-ray transform inverse problem);
\item[2.] 
determining the sound speed in a domain from boundary measurements of solutions of the wave equation (Gelfand problem);
\item[3.] 
determining the electrical conductivity in a domain from voltage and current measurements on its boundary (Calder\'on problem).
\end{enumerate}
In earlier lecture notes published in \cite{Salo_lectures} we discussed these problems from a microlocal point of view when the background geometry was Euclidean. In these notes we will consider geometric versions of these problems, in the more general setting of a Riemannian background geometry. This material is based on lecture notes for two online minicourses, one organized at DTU, Copenhagen, in January 2021 and another at CRM, Montreal, in August 2021.

Microlocal analysis arises in various ways in the above problems. Here are a few examples:
\begin{enumerate}
\item[1.] 
{\bf X-ray transform:} the X-ray transform is a Fourier integral operator (FIO), and under certain conditions its normal operator is an elliptic pseudodifferential operator ($\Psi$DO). Microlocal analysis can be used to predict which sharp features (singularities) of the image can be reconstructed in a stable way from measurements. 
\item[2.] 
{\bf Gelfand problem:} the boundary measurement map (hyperbolic Dirichlet-to-Neumann map) is an FIO. The scattering relation of the sound speed as well as X-ray transforms of the coefficients can be computed from the canonical relation and the symbol of this FIO.
\item[3.] 
{\bf Calder\'on problem:} the boundary measurement map (Dirichlet-to-Neumann map) is a $\Psi$DO. The boundary values of the conductivity as well as its derivatives can be computed from the symbol of this $\Psi$DO.
\end{enumerate}

The above inverse problems are already relevant in Euclidean space. However, the strength of microlocal methods becomes more apparent in \emph{geometric}, or \emph{non-Euclidean}, settings. For X-ray transform problems this will mean that functions are integrated over geodesics instead of straight lines. For Gelfand or Calder\'on type problems this will mean that domains in $\mR^n$ are replaced by more general geometric spaces.

A particularly clean setting, which is still relevant for several applications, is the one where domains in $\mR^n$ are replaced by Riemannian manifolds and straight lines are replaced by geodesic curves of a smooth Riemannian metric. We will focus on this setting and formulate our questions on compact Riemannian manifolds $(M,g)$ with smooth boundary, although Euclidean space will be used for examples. We also remark that analogous \emph{inverse scattering problems} could be considered on noncompact manifolds. 

For reading these notes we assume familiarity with basic Riemannian geometry, partial differential equations, and pseudodifferential operators roughly at the level of the textbooks \cite{Lee} and \cite{Folland}.

\addtocontents{toc}{\SkipTocEntry}
\subsection*{Notation}

In these notes $M$ will always be a compact, oriented, smooth ($=C^{\infty}$) manifold with smooth boundary, and $g$ will be a smooth Riemannian metric on $M$. We assume that $n = \dim(M) \geq 2$. We write $\langle \,\cdot\,, \,\cdot\, \rangle_g$ and $\abs{\,\cdot\,}_g$ for the $g$-inner product and norm on tangent vectors. In local coordinates we write $g = (g_{jk})_{j,k=1}^n$, and $(g^{jk})$ is the inverse matrix of $(g_{jk})$. Thus if $x = (x_1, \ldots, x_n)$ are local coordinates and if $\p_j = \frac{\p}{\p x_j}$ are the corresponding coordinate vector fields, then $g_{jk} = \langle \p_j, \p_k \rangle_g$ and 
\[
\langle X^j \p_j, Y^k \p_k \rangle_g = g_{jk} X^j Y_k, \qquad \abs{X^j \p_j}_g = (g_{jk} X^j X^k)^{1/2}.
\]
Here and below we use the Einstein summation convention that a repeated upper and lower index is summed from $1$ to $n$ (i.e.\ we omit the sum signs).

We denote by $\nabla_g = \mathrm{grad}_g$ and by $\mathrm{div}_g$ the Riemannian gradient and divergence on $M$. The Laplace-Beltrami operator is $\Delta_g = \mathrm{div}_g \nabla_g$. 
In local coordinates one has the formulas 
\begin{align*}
\nabla_g u &= g^{jk} \p_j u \p_k, \\
\mathrm{div}_g(X^j \p_j) &= \det(g)^{-1/2} \p_j(\det(g)^{1/2} X^j), \\
\Delta_g u &= \det(g)^{-1/2} \p_j(\det(g)^{1/2} g^{jk} \p_k u).
\end{align*}
We denote the volume form on $(M,g)$ by $dV_g$, and the induced volume form on $\p M$ by $dS_g$. If $u, v \in C^{\infty}(M)$, one has the integration by parts (or Green) formula 
\[
\int_{\p M} (\p_{\nu} u) v \,dS_g = \int_M ((\Delta_g u) v + \langle \nabla_g u, \nabla_g v \rangle_g) \,dV_g
\]
where $\nu$ is the outer unit normal vector to $\p M$, and $\p_{\nu} u = \langle \nabla_g u, \nu \rangle_g|_{\p M}$ is the normal derivative on $\p M$.

We note that we sometimes drop the subindex $g$ for brevity and write $\abs{v}$ instead of $\abs{v}_g$ etc. All geodesics are assumed to have unit speed, i.e.\ to satisfy $\abs{\dot{\gamma}(t)}_g = 1$.

\addtocontents{toc}{\SkipTocEntry}
\subsection*{Acknowledgements}

This material is based on lecture notes for two online minicourses, one at the DTU PhD Winter School in January 2021 and another at the CRM S\'eminaire de Math\'ematiques Sup\'erieures ``Microlocal analysis: theory and applications'' in August 2021. The author would like to thank the organizers of these schools at DTU and CRM, in particular Kim Knudsen, Katya Krupchyk and Suresh Eswarathasan, for the opportunity to give these lectures. The author would also like to thank all the students who attended the online courses for their questions and comments that have improved the presentation. The author is partly supported by the Academy of Finland (Centre of Excellence in
Inverse Modelling and Imaging, grant 353091) and by the European Research Council under Horizon 2020 (ERC CoG 770924).

\section{Geodesic X-ray transform} \label{section_ray_transform}

In this section we discuss the geodesic X-ray transform, which generalizes the classical X-ray (or Radon) transform in Euclidean space. 

\subsection{Radon transform in $\mR^2$}

To set the stage, we review a few facts about the classical Radon transform following \cite[Chapter 1]{PSU_book}. See \cite{Helgason, Natterer} for further information on the Radon transform.

The \emph{X-ray transform} $If$ of a function $f$ in $\mR^n$ encodes the integrals of $f$ over all straight lines, whereas the \emph{Radon transform} $Rf$ encodes the integrals of $f$ over $(n-1)$-dimensional planes. We will focus on the case $n=2$, where the two transforms coincide. This transform appears naturally in medical imaging in X-ray computed tomography (CT) and positron emission tomography (PET).

There are many ways to parametrize the set of lines in $\mR^2$. We will parametrize lines by their direction vector $\omega$ and signed distance $s$ from the origin.

\begin{definition}
If $f \in C^{\infty}_c(\mR^2)$, the \emph{Radon transform} of $f$ is the function 
\begin{equation*}
Rf(s,\omega) := \int_{-\infty}^{\infty} f(s\omega^{\perp} + t\omega) \,dt, \quad s \in \mR, \ \omega \in S^1.
\end{equation*}
Here $\omega^{\perp}$ is the vector in $S^1$ obtained by rotating $\omega$ counterclockwise by $90^{\circ}$.
\end{definition}

There is a well-known relation between $Rf$ and the Fourier transform $\hat{f}$. We use the convention 
\[
\hat{f}(\xi) = \mF f(\xi) = \int_{\mR^n} e^{-ix \cdot \xi} f(x) \,dx.
\]
We denote by $(Rf)\etilde(\,\cdot\,,\omega)$ the Fourier transform of $Rf$ with respect to $s$.

\begin{theorem}[Fourier slice theorem] \label{thm_fourier_slice}
\begin{equation*}
(Rf)\etilde(\sigma,\omega) = \hat{f}(\sigma \omega^{\perp}).
\end{equation*}
\end{theorem}
\begin{proof}
Parametrizing $\mR^2$ by $y = s\omega^{\perp} + t\omega$, we have 
\begin{align*}
(Rf)\etilde(\sigma,\omega) &= \int_{-\infty}^{\infty} e^{-i\sigma s} \left[ \int_{-\infty}^{\infty} f(s\omega^{\perp} + t\omega) \,dt \right] \,ds = \int_{\mR^2} e^{-i\sigma y \cdot \omega^{\perp}} f(y) \,dy \\
 &= \hat{f}(\sigma \omega^{\perp}). \qedhere
\end{align*}
\end{proof}

This already proves that the Radon transform $Rf$ uniquely determines $f$:

\begin{theorem}[Uniqueness] \label{corollary_radon_injectivity}
If $f \in C^{\infty}_c(\mR^2)$ is such that $Rf \equiv 0$, then $f \equiv 0$.
\end{theorem}
\begin{proof}
If $Rf \equiv 0$, then  $\hat{f} \equiv 0$ by Theorem \ref{thm_fourier_slice} and consequently $f \equiv 0$  by the Fourier inversion theorem.
\end{proof}

To obtain a different inversion method, and for later purposes, we will consider the adjoint of $R$. The formal adjoint of $R$ is the \emph{backprojection operator}\footnote{The formula for $R^*$ is obtained as follows: if $f \in C^{\infty}_c(\mR^2)$, $h \in C^{\infty}(\mR \times S^1)$ one has  
\begin{align*}
(Rf, h)_{L^2(\mR \times S^1)} &= \int_{-\infty}^{\infty} \int_{S^1} Rf(s,\omega) \ol{h(s,\omega)} \,d\omega \,ds \\
 &= \int_{-\infty}^{\infty} \int_{S^1} \int_{-\infty}^{\infty} f(s\omega^{\perp} + t\omega) \ol{h(s,\omega)} \,dt \,d\omega \,ds \\
 &= \int_{\mR^2} f(y) \left( \int_{S^1} \ol{h(y \cdot \omega^{\perp}, \omega)} \,d\omega \right) \,dy.
\end{align*}}
 \begin{equation*}
R^*: C^{\infty}(\mR \times S^1) \to C^{\infty}(\mR^2), \ \ R^* h(y) = \int_{S^1} h(y \cdot \omega^{\perp}, \omega) \,d\omega.
\end{equation*}

The following result shows that the normal operator $R^* R$ is a classical elliptic $\Psi$DO of order $-1$ in $\mR^2$, and also gives an inversion formula.

\begin{theorem} \label{theorem_normal_operator_radon}
(Normal operator) One has 
\begin{equation*}
R^* R = 4\pi \abs{D}^{-1} = \mF^{-1} \left\{ \frac{4\pi}{\abs{\xi}} \mF(\,\cdot\,) \right\},
\end{equation*}
and $f$ can be recovered from $Rf$ by the formula 
\begin{equation*}
f = \frac{1}{4\pi} \abs{D} R^* R f.
\end{equation*}
\end{theorem}
\begin{remark}
Above we have written, for $\alpha \in \mR$, 
\[
\abs{D}^{\alpha} f := \mF^{-1} \{ \abs{\xi}^{\alpha} \hat{f}(\xi) \}.
\]
The notation $(-\Delta)^{\alpha/2} = \abs{D}^{\alpha}$ is also used.
\end{remark}
\begin{proof}
The proof is based on computing $(Rf, Rg)_{L^2(\mR \times S^1)}$ using the Parseval identity, Fourier slice theorem, symmetry and polar coordinates:
\begin{align*}
(R^* R f, g)_{L^2(\mR^2)} &= (Rf, Rg)_{L^2(\mR \times S^1)} \\
 &= \int_{S^1} \left[ \int_{-\infty}^{\infty} (Rf)(s,\omega) \ol{(Rg)(s,\omega)} \,ds \right] \,d\omega \\
 &= \frac{1}{2\pi} \int_{S^1} \left[ \int_{-\infty}^{\infty} (Rf)\etilde(\sigma,\omega) \ol{(Rg)\etilde(\sigma,\omega)} \right] \,d\sigma \,d\omega \\
 &= \frac{1}{2\pi} \int_{S^1} \left[ \int_{-\infty}^{\infty}  \hat{f}(\sigma \omega^{\perp})  \ol{\hat{g}(\sigma \omega^{\perp})} \right] \,d\sigma \,d\omega \\
 &= \frac{2}{2\pi} \int_{S^1} \left[ \int_{0}^{\infty}  \hat{f}(\sigma \omega^{\perp}) \ol{\hat{g}(\sigma \omega^{\perp})} \right] \,d\sigma \,d\omega \\
 &= \frac{2}{2\pi} \int_{\mR^2} \frac{1}{\abs{\xi}} \hat{f}(\xi) \ol{\hat{g}(\xi)} \,d\xi \\
 &= ( 4\pi \mF^{-1} \left\{ \frac{1}{\abs{\xi}} \hat{f}(\xi) \right\}, g)_{L^2(\mR^2)}. \qedhere
\end{align*}
\end{proof}

The same argument, based on computing $(\abs{D_s}^{1/2} Rf, \abs{D_s}^{1/2} Rg)_{L^2(\mR \times S^1)}$ instead of $(Rf, Rg)_{L^2(\mR \times S^1)}$, leads to the famous \emph{filtered backprojection} (FBP) inversion formula:
\begin{equation} \label{fbp_formula}
f = \frac{1}{4\pi} R^* \abs{D_s} R f
\end{equation}
where $\abs{D_s} Rf = \mF^{-1} \{ \abs{\sigma} (Rf) \etilde \}$. This formula is efficient to implement and gives good reconstructions when one has complete X-ray data and relatively small noise, and hence FBP (together with its variants) has been commonly used in X-ray CT scanners.

However, if one is mainly interested in the singularities (i.e.\ jumps or sharp features) of the image, it is possible to use the even simpler \emph{backprojection method}: just apply the backprojection operator $R^*$ to the data $Rf$. Since $R^* R$ is an elliptic $\Psi$DO, singularities are recovered:
\begin{gather*}
\mathrm{sing\,supp}(R^* R f) = \mathrm{sing\,supp}(f), \\
\mathrm{WF}(R^* R f) = \mathrm{WF}(f).
\end{gather*}
Here $\mathrm{sing\,supp}(f)$ and $\mathrm{WF}(f)$ are the singular support and wave front set of $f$, respectively (see e.g.\ \cite[Chapter 1.3.2]{PSU_book}). Moreover, since $R^* R$ is a $\Psi$DO of order $-1$, hence smoothing of order $1$, one expects that $R^* R f$ gives a slightly blurred version of $f$ where the main singularities should still be visible. The ellipticity of the normal operator is also important in the analysis of statistical methods for recovering $f$ from $Rf$ \cite{MonardNicklPaternain}.

The interplay between the Radon and Fourier transforms can further be used to study reconstruction algorithms and stability and range properties for the Radon transform inverse problem. The use of the Fourier transform is possible because the Euclidean space $\mR^2$ is highly symmetric, and can be nicely tiled with straight lines. In more general geometric spaces, symmetries and Fourier methods may not be available and one needs to employ different methods. Some of the available methods will be discussed below.

\subsection{Geodesic X-ray transform}

We will now introduce the geodesic X-ray transform following \cite[Chapters 3 and 4]{PSU_book}, see also \cite{Sh}. This transform appears in seismic and ultrasound imaging, e.g.\ as the linearization of the boundary rigidity/inverse kinematic problem. We will see in the later sections that it also arises in the study of inverse problems for partial differential equations.

Let $(M,g)$ be a compact manifold with smooth boundary, assumed to be embedded in a compact manifold $(N,g)$ without boundary. We parametrize geodesics by points in the \emph{unit sphere bundle}, defined by 
\begin{equation*}
SM := \{ (x,v) \,:\, x \in M, \ v \in T_x M, \ \abs{v}_g = 1 \}.
\end{equation*}
We also consider the unit spheres 
\[
S_x M := \{ v \in T_x M \,:\, \abs{v}_g = 1 \}, \qquad x \in M.
\]
If $(x,v) \in SN$ we denote by $\gamma_{x,v}(t)$ the geodesic in $N$ which starts at the point $x$ in direction $v$, that is, 
\begin{equation*}
D_{\dot{\gamma}} \dot{\gamma} = 0, \quad \gamma_{x,v}(0) = x, \quad \dot{\gamma}_{x,v}(0) = v.
\end{equation*}
Here $D$ denotes the Levi-Civita connection induced by $g$. The geodesic equation $D_{\dot{\gamma}} \dot{\gamma} = 0$ reads in local coordinates as 
\[
\ddot{\gamma}^l(t) + \Gamma_{jk}^l(\gamma(t)) \dot{\gamma}^j(t) \dot{\gamma}^k(t) = 0
\]
where $\Gamma_{jk}^l = \frac{1}{2} g^{lm}(\p_j g_{km} + \p_k g_{jm} - \p_m g_{jk})$ are the Christoffel symbols of the metric $g = (g_{jk})_{j,k=1}^n$, and $(g^{jk})$ is the inverse matrix of $(g_{jk})$.

We also denote by $\varphi_t$ the \emph{geodesic flow} on $SN$, 
\[
\varphi_t: SN \to SN, \ \ \varphi_t(x,v) = (\gamma_{x,v}(t), \dot{\gamma}_{x,v}(t)).
\]
If $(x,v) \in SM$ let $\tau(x,v) \in [0,\infty]$ be the first time when $\gamma_{x,v}(t)$ exits $M$, 
\begin{equation*}
\tau(x,v) := \sup \,\{ t \geq 0 \,:\, \gamma_{x,v}([0,t]) \subset M \}.
\end{equation*}
We assume that $(M,g)$ is \emph{nontrapping}, meaning that $\tau(x,v)$ is finite for any $(x,v) \in SM$. (If $\tau(x,v) = \infty$, we say that the geodesic $\gamma_{x,v}$ is \emph{trapped}.)

\begin{definition}
The \emph{geodesic X-ray transform} of a function $f \in C^{\infty}(M)$ is defined by 
\begin{equation*}
If(x,v) := \int_0^{\tau(x,v)} f(\gamma_{x,v}(t)) \,dt, \quad (x,v) \in \partial(SM).
\end{equation*}
\end{definition}

Thus, $If$ encodes the integrals of $f$ over all maximal geodesics in $M$ starting from $\partial M$, such geodesics being parametrized by points of $\partial(SM) = \{ (x,v) \in SM \,:\, x \in \partial M \}$. We note that $I$ can be extended to act on $L^2(M)$ \cite[Proposition 1.4.2]{PSU_book}.

So far we have not imposed any restrictions on the behavior of geodesics in $(M,g)$ other than the nontrapping condition. However, invertibility of the geodesic X-ray transform is only known under certain geometric restrictions. One class of manifolds where such results have been proved is the following.

\begin{definition}
A compact Riemannian manifold $(M,g)$ with smooth boundary is called \emph{simple} if 
\begin{enumerate}
\item[(a)]
its boundary $\partial M$ is strictly convex,
\item[(b)] 
it is nontrapping, and
\item[(c)]
no geodesic has conjugate points.
\end{enumerate}
\end{definition}

We explain briefly the notions appearing in the definition:
\begin{tehtratk}
\item (Strict convexity)
We say that $\p M$ is \emph{strictly convex} if the second fundamental form of $\p M$ in $M$ is positive definite. This implies in particular that any geodesic in $N$ that is tangent to $\partial M$ stays outside $M$ for small positive and negative times. Thus any maximal geodesic going from $\partial M$ into $M$ stays inside $M$ except for its endpoints, which corresponds to the usual notion of strict convexity in Euclidean space.

We will only use the following consequence of (a): if $\p M$ is strictly convex, then the exit time function $\tau$ is $C^{\infty}$ in $SM^{\mathrm{int}}$ and hence all functions in the analysis below are $C
^{\infty}$, see \cite[Section 3.2]{PSU_book}. In fact assumption (a) can often be removed with extra arguments   \cite{GMT}. \\[-5pt]

\item (Nontrapping)
The \emph{nontrapping} condition means that any geodesic in $M$ should reach the boundary $\p M$ in finite time. An example of a trapped geodesic is the equator in a large spherical cap $\{ x \in S^2 \,:\, x_3 \geq -\eps \}$. \\[-5pt]

\item (Conjugate points)
If $\gamma: [a,b] \to M$ is a geodesic segment and if there is a nontrivial smooth family of geodesics $(\gamma_s)_{s \in (-\eps, \eps)}$ such that $\gamma_0 = \gamma$ and $\gamma_s(a) = \gamma(a)$, $\gamma_s(b) = \gamma(b)$ for $s \in (-\eps,\eps)$, then the points $\gamma(a)$ and $\gamma(b)$ are said to be \emph{conjugate} along $\gamma$. This is a sufficient and almost necessary condition for conjugate points; for precise definitions see \cite[Section 3.7]{PSU_book}. As an example, the north and south poles on the sphere are conjugate along any geodesic (=great circle) connecting them.

Part (c) of the definition of a simple manifold states that there is no pair of conjugate points along any geodesic segment in $M$. Informally this means that there is no family of geodesics that starts at one point and converges to another point after some time. When $\dim(M) = 2$, a sufficient condition for no conjugate points is that the Gaussian curvature satisfies $K(x) \leq 0$ for all $x \in M$ (in higher dimensions it is enough that all sectional curvatures are nonpositive).
\end{tehtratk}

\vspace{5pt}

The class of simple manifolds shows up frequently in geometric inverse problems. We mention that any simple manifold is diffeomorphic to a ball, so one can think of $M$ as being just the closed unit ball in $\mR^n$ with some nontrivial Riemannian metric $g$. There are several equivalent definitions \cite[Section 3.8]{PSU_book} and we will need the following.

\begin{lemma}[Exponential map on simple manifolds] \label{lemma_simple_exp}
Let $(M,g)$ be compact with strictly convex smooth boundary. Then $(M,g)$ is simple iff there is an open manifold $(U,g)$ containing $M$ as a compact subdomain such that for any $p \in M$, the exponential map $\exp_p$ is a diffeomorphism from its maximal domain $D_p$ in $T_p U$ onto $U$. 
\end{lemma}

Recall that $\exp_p: D_p \subset T_p U \to U$ parametrizes part of $U$ by radial geodesics starting at $p$. The proof that any simple manifold satisfies the condition in Lemma \ref{lemma_simple_exp} requires geometric arguments and may be found in \cite[Section 3.8]{PSU_book}. It follows that any $x \in U$ can be uniquely written as 
\begin{equation} \label{riemannian_polar_first}
x = \exp_p(r\omega)
\end{equation}
for some $r \geq 0$ and $\omega \in S^{n-1}$, with $r \omega \in D_p$. Thus we may identify $x \in U$ with $(r,\omega)$. The coordinates $(r,\omega)$ are called \emph{Riemannian polar coordinates}, or \emph{polar normal coordinates}, in $(U,g)$. Thus Lemma \ref{lemma_simple_exp} essentially states that a manifold is simple iff it admits global polar coordinates centered at any point.

\begin{example}[Simple manifolds]
Strictly convex bounded smooth domains in $\mR^n$, or in nonpositively curved Riemannian manifolds, are simple. An example with positive curvature is given by the small spherical cap $M = \{ x \in S^2 \,:\, x_3 \geq \eps \}$, where $S^2$ is the unit sphere in $\mR^3$ and $\eps > 0$. Note that such a spherical cap does not contain trapped geodesics or conjugate points. Small metric perturbations of simple manifolds are also simple.
\end{example}

The main result in this setting, proved first in \cite{Mu1} in two dimensions, states that the geodesic X-ray transform is injective on simple manifolds.

\begin{theorem}[Injectivity] \label{thm_xray}
Let $(M,g)$ be a simple manifold. If $f \in C^{\infty}(M)$ satisfies $If = 0$, then $f = 0$.
\end{theorem}

We note that on general manifolds injectivity may fail:

\begin{example}[Counterexamples]
There are two basic examples of manifolds where the geodesic X-ray transform is not injective. The first is a large spherical cap $M = \{ x \in S^2 \,:\, x_3 \geq -\eps \}$. Any odd function $f$ supported in a small neighborhood of $e_1$ and $-e_1$ integrates to zero over all great circles, hence $If = 0$ but $f$ is nontrivial. Another example is a catenoid type surface with a flat cylinder glued in the middle \cite[Section 2.5]{PSU_book}. Note that both examples contain trapped geodesics. The latter example has no conjugate points.
\end{example}

We mention that the nontrapping condition can be replaced by hyperbolic trapped set \cite{Guillarmou}. When $\dim(M) \geq 3$ further injectivity results are available, based on the microlocal method introduced in \cite{UhlmannVasy}. These results are valid on strictly convex nontrapping manifolds that admit a strictly convex function, i.e.\ a function $\varphi \in C^{\infty}(M)$ such that $\mathrm{Hess}_g(\varphi) > 0$, or more generally are foliated by strictly convex hypersurfaces. Such manifolds may have conjugate points. 

The following questions remain open (see e.g.\ \cite{IlmavirtaMonard, PSU_book} for further references):

\begin{question}
Is the geodesic X-ray transform injective on compact strictly convex nontrapping manifolds?
\end{question}

\begin{question}
Is the local geodesic X-ray transform injective on strictly convex 2D manifolds? This is true when $(M,g)$ is real-analytic \cite{SU_nonsimple, MST} or when $\dim(M) \geq 3$ \cite{UhlmannVasy}.
\end{question}

\begin{question}
Does every simple manifold admit a strictly convex function?
\end{question}

\begin{question}
Are there other interesting examples of manifolds where the geodesic X-ray transform is not injective?
\end{question}

We will sketch a proof of Theorem \ref{thm_xray} in the end of this section. However, we first discuss some microlocal aspects of the geodesic X-ray transform.

\subsection{Microlocal aspects}

When $(M,g)$ is compact, strictly convex and nontrapping, it can be proved that $I$ is a Fourier integral operator in $M^{\mathrm{int}}$ (see e.g.\ \cite{MonardStefanovUhlmann}). For general manifolds it is not reasonable to expect exact inversion formulas for $I$ like the FBP formula \eqref{fbp_formula} in the Euclidean case. However, if we additionally assume that $(M,g)$ is simple, an analogue of Theorem \ref{theorem_normal_operator_radon} persists:

\begin{theorem}[Normal operator] \label{thm_xray_normal}
Let $(M,g)$ be a simple manifold. Then $I^* I$, computed with respect to suitable $L^2$ inner products, is a classical elliptic $\Psi$DO of order $-1$ in $M^{\mathrm{int}}$.
\end{theorem}
\begin{proof}
(Sketch, see \cite[Section 8.1]{PSU_book} for details.) The idea of the proof, as in Theorem \ref{theorem_normal_operator_radon}, is to compute the inner product $(If, Ih)$ in a suitable $L^2$ inner product on $\p SM$. The Fourier slice theorem is not available, but one can use the definitions and directly express the normal operator as 
\[
I^* I f(x) = 2 \int_{D_x} \frac{f(\exp_x(w))}{\abs{w}_g^{n-1}} \,dT_x(w)
\]
where $D_x$ is the maximal domain of $\exp_x$ in $T_x M$. This works on any nontrapping manifold. Now we invoke the simplicity assumption (Lemma \ref{lemma_simple_exp}) which guarantees that one has global polar coordinates $y = \exp_x(w)$ on $M$. This is somewhat analogous to the Euclidean polar coordinates in Theorem \ref{theorem_normal_operator_radon}. Thus 
\[
I^* I f(x) = 2 \int_M \frac{a(x,y)}{d_g(x,y)^{n-1}} f(y) \,dV_g(y)
\]
where $d_g(x,y)$ is the $g$-distance between $x$ and $y$, and $a(x,y)$ is a smooth positive function with $a(x,x) = 1$. We have now computed the Schwartz kernel of $I^* I$, and this kernel is smooth away from the diagonal and has a singularity of the form $d_g(x,y)^{1-n}$ on the diagonal. It follows that $I^* I$ is a classical $\Psi$DO of order $-1$ and its principal symbol is $c_n \abs{\xi}_g^{-1}$, showing that $I^* I$ is elliptic.
\end{proof}

Theorem \ref{thm_xray_normal}, applied in an extension of $M$, shows that on simple manifolds 
\begin{gather*}
\mathrm{sing\,supp}(I^* I f) = \mathrm{sing\,supp}(f), \\
\mathrm{WF}(I^* I f) = \mathrm{WF}(f).
\end{gather*}
Thus we can at least determine the singularities (i.e.\ jumps etc) of $f$ from the knowledge of $If$. Since $I^* I$ is an elliptic $\Psi$DO, the standard parametrix construction implies that it can be inverted modulo a (compact) smoothing operator. This is not in general sufficient for showing that $I$ is honestly invertible. However, in the following situations we do get injectivity of $I$:

\begin{tehtratk}
\item[1.]
$(M,g)$ is real-analytic. The argument is based on \emph{analytic microlocal analysis} and one proof proceeds roughly as follows (see \cite{SU_generic} for details): now $I^* I$ is an analytic elliptic $\Psi$DO, and it has a parametrix $Q$ so that 
\[
Q(I^* I f) = f + Rf
\]
where $R$ is an \emph{analytic smoothing operator} (i.e.\ it maps any function to a real-analytic function). If $If = 0$, it follows that $f = -Rf$ is real-analytic. Moreover, if $If = 0$ one can do a boundary determination argument to show that $f$ must vanish to infinite order on $\p M$. Combining these facts proves that $f=0$. An alternative proof, based on considering $I$ directly as an analytic FIO, may be found in \cite{SU_nonsimple, MST}.
\item[2.] 
$n \geq 3$ and $(M,g)$ is foliated by strictly convex hypersurfaces \cite{UhlmannVasy}. In this case the \emph{localized} normal operator $(\chi I)^* (\chi I)$, where $\chi$ is a cutoff localizing to all sufficiently short geodesics near a fixed point $x_0 \in \p M$, is elliptic (this fails when $n = 2$). A modification of this idea, which involves a suitable artifical ``boundary at infinity'' near $x_0$ and conjugation by certain exponentials, leads to an elliptic $\Psi$DO in Melrose's scattering calculus which is honestly invertible (since after applying a parametrix, the related smoothing operator becomes small in norm by adjusting a parameter controlling the artificial boundary). Thus $If = 0$ implies $f=0$ near $x_0$. Iterating this result by using the strictly convex foliation implies that $f=0$ everywhere. An alternative version of this argument may be found in \cite{Vasy_semiclassical}.
\end{tehtratk}

The microlocal ideas above are not sufficient to prove Theorem \ref{thm_xray} in general, but one can use energy methods instead.

\subsection{Proof of injectivity}

In the rest of this section we will sketch a proof of Theorem \ref{thm_xray} following the argument in \cite{PSU1} under two simplifying assumptions:
\begin{itemize}
\item 
$\dim(M) = 2$ (to simplify the analysis on $SM$);
\item 
$f \in C^{\infty}_c(M^{\mathrm{int}})$ (to remove regularity issues near $\p M$).
\end{itemize}
The proof contains two parts:
\begin{enumerate}
\item[1.]
Reduction from the integral equation $If = 0$ into a partial differential equation $VXu = 0$ on $SM$.
\item[2.]
Uniqueness result for the equation $VXu = 0$ in $SM$ based on energy methods.
\end{enumerate}
A more detailed presentation may be found in \cite[Chapter 4]{PSU_book}.

\subsubsection{Reduction to PDE}

Assume that $f \in C^{\infty}_c(M^{\mathrm{int}})$ satisfies $If = 0$. We begin by introducing the primitive function 
\[
u(x,v) = u^f(x,v) := \int_0^{\tau(x,v)} f(\varphi_t(x,v)) \,dt, \qquad (x,v) \in SM.
\]
Here we think of $f$ as a function on $SM$ by taking $f(x,v) = f(x)$. Note that $u|_{\p(SM)} = If = 0$. Since $\tau$ is smooth in $SM^{\mathrm{int}}$ and $f$ vanishes near $\p M$, we in fact have $u \in C^{\infty}_c(SM^{\mathrm{int}})$.

Next we introduce the \emph{geodesic vector field} $X: C^{\infty}(SN) \to C^{\infty}(SN)$, which differentiates a function on $SN$ along geodesic flow:
\[
Xw(x,v) = \frac{d}{ds} w(\varphi_s(x,v)) \Big|_{s=0}.
\]
We note that the function $u = u^f$ above satisfies 
\begin{align*}
Xu(x,v) &= \frac{d}{ds} u(\varphi_s(x,v)) \Big|_{s=0} = \frac{d}{ds} \int_0^{\tau(\varphi_s(x,v))} f(\varphi_t(\varphi_s(x,v))) \,dt \Big|_{s=0} \\
 &= \frac{d}{ds} \int_0^{\tau(x,v)-s} f(\varphi_{t+s}(x,v)) \,dt \Big|_{s=0} \\
 &= \frac{d}{ds} \int_s^{\tau(x,v)} f(\varphi_r(x,v)) \,dr \Big|_{s=0} \\
 &= -f(x).
\end{align*}
In particular we have 
\begin{equation} \label{isp}
Xu = -f(x) \text{ on $SM$}, \qquad u|_{\p SM} = If = 0.
\end{equation}

The problem \eqref{isp} can be considered as an \emph{inverse source problem} for a transport equation: the source $f(x)$ in the equation produces a measurement $u|_{\p(SM)} = If = 0$. We wish to prove uniqueness in the sense that if the measurement $u|_{\p(SM)}$ is zero, then the source must be zero.

Note that the equation is on $SM = \{ (x,v) \in TM \,:\, \abs{v} = 1 \}$, but the source $f(x)$ has the special property that it only depends on $x$ and not on $v$. We can further get rid of the source by differentiating the equation $Xu(x,v) = -f(x)$ with respect to $v$. To do this in a coordinate-invariant way, we introduce the following notions:

\begin{definition}
Let $(M,g)$ be an oriented two-dimensional manifold. Given $v \in S_x M$, we define $v^{\perp}$ (\emph{rotation by $90^{\circ}$ counterclockwise}) to be the unique vector in $S_x M$ so that $(v, v^{\perp})$ is a positively oriented orthonormal basis of $T_x M$. Morever, given $\theta \in \mR$, we define the \emph{rotation} 
\[
R_{\theta} v = (\cos \theta) v + (\sin \theta) v^{\perp}.
\]
Finally, we define the \emph{vertical vector field} $V: C^{\infty}(SM) \to C^{\infty}(SM)$ by 
\[
Vw(x,v) = \frac{d}{d\theta} w(x, R_{\theta} v) \Big|_{\theta=0}, \qquad (x,v) \in SM.
\]
\end{definition}

\begin{example}[$X$ and $V$ in the Euclidean disk] \label{ex_xv}
Let $M = \ol{\mathbb{D}} \subset \mR^2$ and let $g$ be the Euclidean metric. Then 
\[
SM = \{ (x,v_{\theta}) \,:\, x \in M, \ \theta \in (-\pi,\pi] \}
\]
where $v_{\theta} = (\cos \theta, \sin \theta)$. We identify $(x,v_{\theta})$ with $(x,\theta)$. Then 
\[
Xw(x,\theta) = \frac{d}{dt} w(x+t v_{\theta}, \theta) \Big|_{t=0} = v_{\theta} \cdot \nabla_x w(x, \theta)
\]
and 
\[
Vw(x,\theta) = \frac{d}{d\theta} w(x, \theta).
\]
\end{example}

If $f(x)$ is independent of $v$, clearly $Vf = 0$. Thus if $f \in C^{\infty}_c(M^{\mathrm{int}})$ satisfies $If = 0$, then by \eqref{isp} the primitive $u = u^f \in C^{\infty}_c(SM^{\mathrm{int}})$ satisfies 
\[
VXu = 0 \text{ in $SM$}.
\]
This reduces the geodesic X-ray transform problem to showing that the only solution of the equation $VXu = 0$ on $SM$ which vanishes near $\p M$ is the zero solution.

\subsubsection{Uniqueness via energy methods}

The required uniqueness result will be a consequence of the following energy estimate.

\begin{proposition}[Energy estimate] \label{prop_energy}
If $(M,g)$ is a two-dimensional simple manifold, then 
\[
\norm{Xu}_{L^2(SM)} \leq \norm{VXu}_{L^2(SM)}
\]
for any $u \in C^{\infty}_c(SM^{\mathrm{int}})$.
\end{proposition}

The $L^2$ norm above is interpreted as follows. Recall that on any Riemannian manifold $(M,g)$ there is a volume form $dV_g$. Moreover, if $x \in M$ the metric $g$ induces an inner product (i.e.\ metric) $g(x)$ on $T_x M$, and hence a metric and volume form $dS_x$ on the unit sphere $S_x M$. We then have the $L^2(SM)$ inner product 
\[
(u, w)= \int_{SM} u\bar{w} \,d\Sigma := \int_M \int_{S_x M} u(x,v) \ol{w(x,v)} \,dS_x(v) \,dV_g(x)
\]
and the corresponding norm 
\[
\norm{u} = \norm{u}_{L^2(SM)} = \left( \int_{SM} \abs{u}^2 \,d\Sigma \right)^{1/2}.
\]

The proof of the main theorem, when $\dim(M) = 2$ and $f \in C^{\infty}_c(M^{\mathrm{int}})$, follows easily from Proposition \ref{prop_energy}.

\begin{proof}[Proof of Theorem \ref{thm_xray}]
Let $f \in C^{\infty}_c(M^{\mathrm{int}})$ satisfy $If = 0$. We have seen that the primitive $u = u^f$ is in $C^{\infty}_c(SM^{\mathrm{int}})$ and satisfies $VXu = 0$ in $SM$. Proposition \ref{prop_energy} gives $Xu = 0$ in $SM$. By \eqref{isp} we get $f = -Xu = 0$.
\end{proof}

It remains to prove Proposition \ref{prop_energy}. Write 
\[
P := VX.
\]
The equation $Pu = 0$ in $SM$ is a second order PDE on the three-dimensional manifold $SM$. It does not belong to any of the standard classes (elliptic, parabolic, hyperbolic etc). Nevertheless we can prove an energy estimate for it by using a \emph{positive commutator argument}.

We first need to compute the formal adjoint of $P$ in the $L^2(SM)$ inner product. We start with the adjoints of $X$ and $V$.

\begin{lemma}[Adjoints of $X$ and $V$] \label{lemma_adjoints}
The vector fields $X$ and $V$ are formally skew-adjoint operators in the sense that 
\[
(Xu, w) = -(u, Xw), \quad (Vu, w) = -(u, Vw)
\]
for $u, w \in C^{\infty}_c(SM^{\mathrm{int}})$.
\end{lemma}

Assuming this, the formal adjoint of $P$ is $P^* = (VX)^* = XV$. Thus we may decompose $P$ in terms of its self-adjoint and skew-adjoint parts:
\begin{equation} \label{p_decomposition}
P = A+iB, \qquad A = \frac{P+P^*}{2}, \qquad B = \frac{P-P^*}{2i}.
\end{equation}
(Compare with the decomposition $z = a+ib$ of a complex number into its real and imaginary parts.) Since $A^* = A$ and $B^* = B$, we can now study the norm $\norm{VXu} = \norm{Pu}$ for $u \in C^{\infty}_c(SM^{\mathrm{int}})$ as follows:
\begin{align}
\norm{Pu}^2 &= (Pu, Pu) = ((A+iB)u, (A+iB)u) \notag \\
 &= \norm{Au}^2 + \norm{Bu}^2 + i(Bu, Au) - i(Au,Bu) \notag \\
 &= \norm{Au}^2 + \norm{Bu}^2 + (i[A,B]u, u) \label{au_bu_formula}
\end{align}
where $[A,B] := AB-BA$ is the \emph{commutator} of $A$ and $B$.

In Proposition \ref{prop_energy} we need to prove that $\norm{Pu} \geq \norm{Xu}$. We can obtain a lower bound for $\norm{Pu}$ from \eqref{au_bu_formula} if the commutator term $(i[A,B]u, u)$ is positive (or if it can be absorbed in the positive terms $\norm{Au}^2$ and $\norm{Bu}^2$). The commutator has the form 
\begin{align*}
2i[A,B] &= \frac{1}{2} [P+P^*, P-P^*] = [P^*, P] = P^* P - P P^* \\
 &= XVVX - VXXV.
\end{align*}
To study $[A,B]$ we need to commute $X$ and $V$. Define the vector field 
\[
X_{\perp} := [X, V].
\]

\begin{lemma}[Commutator formulas] \label{lemma_commutator}
If $(M,g)$ is two-dimensional, one has 
\begin{align*}
[X,V] &= X_{\perp}, \\
[V,X_{\perp}] &= X, \\
[X, X_{\perp}] &= -KV
\end{align*}
where $K$ is the Gaussian curvature of $(M,g)$.
\end{lemma}

\begin{example}[Euclidean case]
Let $M = \ol{\mathbb{D}} \subset \mR^2$ and let $g$ be the Euclidean metric. As in Example \ref{ex_xv} we may identify $(x,v_{\theta})$ with $(x,\theta)$. Then $X_{\perp}$ has the form
\begin{align*}
X_{\perp} w &= XVw - VXw = v_{\theta} \cdot \nabla_x(\p_{\theta} w) - \p_{\theta}(v_{\theta} \cdot \nabla_x w) \\
 &= -(\p_{\theta} v_{\theta}) \cdot \nabla_x w = -v_{\theta}^{\perp} \cdot \nabla_x w.
\end{align*}
The formulas in Lemma \ref{lemma_commutator} can be checked by direct computations, e.g. 
\begin{align*}
[X, X_{\perp}] w &= XX_{\perp}w - X_{\perp} Xw = v_{\theta} \cdot \nabla_x(-v_{\theta}^{\perp} \cdot \nabla_x w) + v_{\theta}^{\perp} \cdot \nabla_x (v_{\theta} \cdot \nabla_x w) \\
 &= 0.
\end{align*}
This is consistent since $K=0$ for the Euclidean metric. For a general metric, computing $[X, X_{\perp}]$ requires commuting two covariant derivatives, and hence one expects the curvature to appear.
\end{example}

We will indicate how to prove Lemmas \ref{lemma_adjoints} and \ref{lemma_commutator} in the end of this section. Using Lemma \ref{lemma_commutator}, we can easily compute the commutator $i[A,B]$:
\begin{align*}
2i[A,B] &= XVVX - VXXV \\
 &= VXVX + X_{\perp} VX - VXVX - VXX_{\perp} \\
 &= X_{\perp} VX - VXX_{\perp} \\
 &= VX_{\perp} X - XX - VX X_{\perp} \\
 &= VKV - XX.
\end{align*}
Thus by Lemma \ref{lemma_adjoints} 
\begin{equation} \label{ab_formula}
(2i[A,B]u, u) = \norm{Xu}^2 - (KVu, Vu).
\end{equation}
We observe:
\begin{itemize}
\item 
If $g$ is the Euclidean metric, then one has $K \equiv 0$ and consequently $(i[A,B]u, u) = \norm{Xu}^2 \geq 0$.
\item 
More generally if $(M,g)$ has nonpositive curvature, i.e.\ $K \leq 0$, then $(i[A,B]u, u) \geq \norm{Xu}^2 \geq 0$.
\end{itemize}
Going back to \eqref{au_bu_formula} and using that $\norm{Au}^2 + \norm{Bu}^2 \geq 0$, we see that if $(M,g)$ is a two-dimensional simple manifold which additionally has nonpositive curvature, then 
\[
\norm{VXu}^2 \geq \norm{Xu}^2, \qquad u \in C^{\infty}_c(SM^{\mathrm{int}}).
\]
This proves Proposition \ref{prop_energy} in the (already nontrivial and interesting) case where $K \leq 0$.

To prove Proposition \ref{prop_energy} in general we need to exploit the $\norm{Au}^2$ and $\norm{Bu}^2$ terms more carefully. Using \eqref{p_decomposition} it is easy to check that 
\begin{align*}
\norm{Au}^2 + \norm{Bu}^2 &= \frac{1}{4} \norm{(P+P^*)u}^2 + \frac{1}{4} \norm{(P-P^*)u}^2 \\
 &= \frac{1}{2} \norm{Pu}^2 + \frac{1}{2} \norm{P^* u}^2.
\end{align*}
Inserting this back in \eqref{au_bu_formula} gives 
\[
\norm{Pu}^2 = \norm{P^* u}^2 + 2(i[A,B]u, u).
\]
Since $P = VX$ and $P^* = XV$, using \eqref{ab_formula} yields the identity
\[
\norm{VXu}^2 = \norm{XVu}^2 - (KVu, Vu) + \norm{Xu}^2.
\]
The identity that we have just proved is an important energy identity in the study of X-ray transforms, known as the \emph{Pestov identity}.

\begin{proposition}[Pestov identity]
If $(M,g)$ is a compact 2D Riemannian manifold with smooth boundary, then for any $u \in C^{\infty}_c(SM^{\mathrm{int}})$ one has 
\[
\norm{VXu}^2 = \norm{XVu}^2 - (KVu, Vu) + \norm{Xu}^2.
\]
\end{proposition}

The proof of Proposition \ref{prop_energy} is completed by the following lemma, which explicitly uses the no conjugate points assumption.

\begin{lemma}
If $(M,g)$ is a two-dimensional simple manifold, then 
\[
\norm{XVu}^2 - (KVu, Vu) \geq 0, \qquad u \in C^{\infty}_c(SM^{\mathrm{int}}).
\]
\end{lemma}
\begin{proof}
If $\gamma: [0,\tau] \to M$ is a geodesic segment, we recall the \emph{index form} (see \cite[Section 3.7]{PSU_book})
\[
I_{\gamma}(Y, Y) = \int_0^{\tau} (\abs{D_t Y(t)}_g^2 - K(\gamma(t)) \abs{Y(t)}_g^2) \,dt
\]
defined for vector fields $Y$ along $\gamma$ that are normal to $\dot{\gamma}$. This is the bilinear form associated with the Jacobi equation $-D_t^2 J(t) - K(\gamma(t)) J(t) = 0$. The basic property is that $\gamma$ has no conjugate points iff $I_{\gamma}(Y, Y) > 0$ for all normal vector fields $Y \not \equiv 0$ along $\gamma$ that vanish at the endpoints.

We will also need the \emph{Santal\'o formula} (see \cite[Section 3.5]{PSU_book}), which is a change of variables formula on $SM$ and states that 
\[
\int_{SM} w \,d\Sigma = \int_{\p_+ SM} \left[ \int_0^{\tau(x,v)} w(\varphi_t(x,v)) \,dt \right] \mu \,d(\p SM)
\]
where $\p_+ SM = \{ (x,v) \in \p(SM) \,:\, \langle v, \nu \rangle_g \leq 0 \}$ and $\mu = -\langle v, \nu \rangle_g$, with $\nu$ being the outward unit normal to $\p M$.
Applying the Santal\'o formula to $w = \abs{XVu}^2 - K \abs{Vu}^2$, and using for any $(x,v) \in \p_+ SM$ the normal vector field 
\[
Y_{x,v}(t) := Vu(\varphi_t(x,v)) \dot{\gamma}(t)^{\perp}
\]
along $\gamma_{x,v}$, implies that 
\begin{align*}
 &\norm{XVu}^2 - (KVu, Vu) \\
 &= \int_{\p_+ SM} \left[ \int_0^{\tau(x,v)} (\abs{XVu(\varphi_t(x,v))}^2 - K(\gamma_{x,v}(t)) \abs{Vu(\varphi_t(x,v))}^2) \,dt \right] \mu \,d(\p SM) \\
 &= \int_{\p_+ SM} \left[ \int_0^{\tau(x,v)} (\abs{D_t Y_{x,v}(t)}^2 - K(\gamma_{x,v}(t)) \abs{Y_{x,v}(t)}^2) \,dt \right] \mu \,d(\p SM) \\
 &=  \int_{\p_+ SM} I_{\gamma_{x,v}}(Y_{x,v}, Y_{x,v}) \mu \,d(\p SM).
\end{align*}
The last quantity is $\geq 0$, since the index form is nonnegative by the no conjugate points condition.
\end{proof}

\begin{remark}
If $(M,g)$ is simple and $n = \dim(M) \geq 3$, the same scheme as above can be used to prove that the geodesic X-ray transform is injective. However, the vector fields $V$ and $X_{\perp}$ need to be replaced by suitable vertical and horizontal gradient operators $\vd$ and $\hd$, and the Pestov identity takes the form 
\[
\norm{\vd Xu}^2 = \norm{X \vd u}^2 - (R \vd u, \vd u) + (n-1) \norm{Xu}^2
\]
where $R Z(x,v) := R_x(Z, v) v$ is the Riemann curvature tensor. We refer the reader to \cite[Section 4.7]{PSU_hd} for more details.
\end{remark}

Finally we discuss the proof of Lemmas \ref{lemma_adjoints} and \ref{lemma_commutator}. One way to prove them is via local coordinate computations. There is a particularly useful coordinate system for this, known as \emph{isothermal coordinates}. The existence of global isothermal coordinates is part of the uniformization theorem for Riemann surfaces. It boils down to the following generalization of the Riemann mapping theorem from simply connected planar domains to simply connected Riemann surfaces. Here we use the basic fact that any simple manifold is diffeomorphic to a ball and hence simply connected \cite[Section 3.8]{PSU_book}.

\begin{theorem}[Global isothermal coordinates] \label{thm_isothermal}
Let $(M,g)$ be a compact oriented simply connected two-dimensional manifold with smooth boundary. There are global coordinates $x = (x_1, x_2)$ on $M$ so that in these coordinates the metric has the form 
\[
g_{jk}(x) = e^{2\lambda(x)} \delta_{jk}
\]
for some real $\lambda \in C^{\infty}(M)$.
\end{theorem}

The isothermal coordinates induce global coordinates $(x_1, x_2, \theta)$ on $SM$ where $\theta \in (-\pi, \pi]$ is the angle between $v$ and $\p/\p x_1$, i.e. 
\[
v = e^{-\lambda(x)}(\cos \theta \frac{\p}{\p x_1} + \sin \theta \frac{\p}{\p x_2}).
\]
We conclude this chapter with exercises that in particular contain the proof of Lemmas \ref{lemma_adjoints} and \ref{lemma_commutator}.

\begin{exercise}
Prove the following stability result for the Radon transform in $\mR^2$:
\[
\norm{f}_{L^2(\mR^2)} \leq \frac{1}{\sqrt{2}} \norm{Rf}_{H^{1/2}_T(\mR \times S^1)}, \qquad f \in C^{\infty}_c(\mR^2),
\]
where we use the norm $\norm{Rf}_{H^{s}_T(\mR \times S^1)} = \norm{(1+\sigma^2)^{s/2} (Rf)\etilde(\sigma, \omega)}_{L^2(\mR \times S^1)}$.
\end{exercise}

\begin{exercise}
Let $(M,g)$ be a compact oriented simply connected two-dimensional manifold with smooth boundary. Use the $(x_1, x_2)$ and $(x_1, x_2, \theta)$ coordinates above to do the following (see \cite[Section 3.5]{PSU_book} for hints if needed):
\begin{enumerate}
\item[(a)] 
Compute the Christoffel symbols $\Gamma_{jk}^l(x)$.
\item[(b)]
Show that $X$, $X_{\perp}$ and $V$ are given by 
\begin{align*}
X &= e^{-\lambda}\left(\cos\theta\frac{\partial}{\partial x_{1}}+
\sin\theta\frac{\partial}{\partial x_{2}}+
\left(-\frac{\partial \lambda}{\partial x_{1}}\sin\theta+\frac{\partial\lambda}{\partial x_{2}}\cos\theta\right)\frac{\partial}{\partial \theta}\right), \\
X_{\perp} &= -e^{-\lambda}\left(-\sin\theta\frac{\partial}{\partial x_{1}}+
\cos\theta\frac{\partial}{\partial x_{2}}-
\left(\frac{\partial \lambda}{\partial x_{1}}\cos\theta+\frac{\partial \lambda}{\partial x_{2}}\sin\theta\right)\frac{\partial}{\partial \theta}\right), \\
V &= \frac{\partial}{\partial\theta}.
\end{align*}
\emph{Hint.} To compute $X$, you can use the equation $\tan \theta(t) = \frac{\dot{x}_2(t)}{\dot{x}_1(t)}$ where $(x_1(t), x_2(t), \theta(t))$ is a geodesic in the $(x_1, x_2, \theta)$ coordinates.
\item[(c)] 
Prove Lemma \ref{lemma_adjoints}. You can use (b) and the fact that 
\[
\int_{SM} w \,d\Sigma = \int_M \int_{-\pi}^{\pi} w(x,\theta) e^{2\lambda(x)} \,d\theta \,dx.
\]
\item[(d)] 
Prove Lemma \ref{lemma_commutator}. You can use (b) and the fact that if $g_{jk}(x) = e^{2\lambda(x)} \delta_{jk}$, then the Gaussian curvature has the form 
\[
K = -\Delta_g \lambda = -e^{-2\lambda}(\p_1^2 \lambda + \p_2^2 \lambda).
\]
\end{enumerate}
\end{exercise}



\section{Gelfand problem} \label{sec_gelfand}

Seismic imaging gives rise to various inverse problems related to determining interior properties, e.g.\ oil deposits or deep structure, of the Earth. Often this is done by using acoustic or elastic waves. We will consider the following problem proposed in \cite{Gelfand}. This problem has many names and equivalent forms and it is also known as the \emph{inverse boundary spectral problem} \cite{KKL} or the \emph{Lorentzian Calder\'on problem} \cite{AFO}.

\begin{quote}
{\bf Gelfand problem:} Is it possible to determine the interior structure of Earth by controlling acoustic waves and measuring vibrations at the surface?
\end{quote}

In seismic imaging one often tries to recover an unknown sound speed. However, in this presentation we consider the simpler case where the sound speed is known and one attempts to recover an unknown potential $q$. We assume that the Earth is modelled by a compact Riemannian $n$-manifold $(M,g)$ with smooth boundary (in practice $M$ is a closed ball in $\mR^3$), and the metric $g$ models the sound speed. In fact, if $c(x)$ is a scalar sound speed in a domain in $\mR^n$, the corresponding metric is 
\[
g_{jk}(x) = c(x)^{-2} \delta_{jk}.
\]
A general metric $g$ corresponds to an \emph{anisotropic} (non-scalar) sound speed. Thus Riemannian geometry already appears when considering sound speeds in Euclidean domains.

Consider the free wave operator 
\[
\Box := \p_t^2 - \Delta
\]
in $M \times (0,T)$, where $\Delta$ is the Laplace-Beltrami operator in $(M,g)$:
\[
\Delta u = \mathrm{div} (\nabla u) = \det(g)^{-1/2} \p_j(\det(g)^{1/2} g^{jk} \p_k u).
\]
Here the operators $\nabla = \nabla_g$, $\mathrm{div} = \mathrm{div}_g$, and $\Delta = \Delta_g$ only act in the $x$ variable. Let $q \in C^{\infty}_c(M^{\mathrm{int}})$ be a time-independent real valued potential.

We assume that the medium is at rest at time $t=0$ and that we take measurements until time $T > 0$. If we prescribe the amplitude of the wave to be $f(x,t)$ on $\p M \times (0,T)$, this leads to a solution $u$ of the wave equation 
\begin{equation} \label{wavedp}
\left\{ \begin{array}{rll}
(\Box + q) u &\!\!\!= 0 & \quad \text{in } M \times (0,T), \\
u &\!\!\!= f & \quad \text{on } \partial M \times (0,T), \\
u = \p_t u &\!\!\!= 0 & \quad \text{on } \{ t = 0 \}.
\end{array} \right.
\end{equation}
Given any $f \in C^{\infty}_c(\p M \times (0,T))$, this initial-boundary value problem has a unique solution $u \in C^{\infty}(M \times (0,T))$ (see \cite[Theorem 7 in \S 7.2.3]{Evans} for the Euclidean case; the proof in the Riemannian case is the same). We assume that we can measure the normal derivative $\p_{\nu} u|_{\p M \times (0,T)}$, where $\p_{\nu} u(x,t) = \langle \nabla u(x,t), \nu(x) \rangle$ and $\nu$ is the outer unit normal to $\p M$. We do such measurements for many different functions $f$.

The ideal boundary measurements in our inverse problem are therefore encoded by the \emph{hyperbolic Dirichlet-to-Neumann map} (DN map for short) 
\[
\Lambda_q: C^{\infty}_c(\p M \times (0,T)) \to C^{\infty}(\p M \times (0,T)), \ \ \Lambda_q(f) = \p_{\nu} u|_{\p M \times (0,T)}.
\]
The Gelfand problem for this model amounts to recovering $q$ from the knowledge of the map $\Lambda_q$.

The following is our main result for the Gelfand problem. For simplicity we assume that the potentials are compactly supported in $M^{\mathrm{int}}$.

\begin{theorem}[Recovering the X-ray transform] \label{thm_gelfand_xray}
Let $(M,g)$ be compact with smooth boundary, let $T > 0$, and assume that $q_1, q_2 \in C^{\infty}_c(M^{\mathrm{int}})$. If $\Lambda_{q_1} = \Lambda_{q_2}$, then 
\begin{equation} \label{wave_xray}
\int_0^{\ell} q_1(\gamma(t)) \,dt = \int_0^{\ell} q_2(\gamma(t)) \,dt
\end{equation}
whenever $\gamma: [0,\ell] \to M$ is a non-trapped maximal geodesic in $M$ with $\ell < T$.
\end{theorem}

If $(M,g)$ is simple, we can combine the above result with injectivity of the geodesic X-ray transform (Theorem \ref{thm_xray}) to obtain a uniqueness result:

\begin{theorem}[Uniqueness] \label{thm_gelfand_uniqueness}
Assume that $(M,g)$ is simple, and let $T > 0$ be larger than the length of the longest maximal geodesic in $M$. If $q_1, q_2 \in C^{\infty}_c(M^{\mathrm{int}})$ and 
\[
\Lambda_{q_1} = \Lambda_{q_2},
\]
then $q_1 = q_2$ in $M$.
\end{theorem}
\begin{proof}
The assumption on $T$ together with Theorem \ref{thm_gelfand_xray} imply that the integrals of $q_1$ and $q_2$ along any maximal geodesic are the same, i.e.\ $I q_1 = I q_2$ where $I$ is the geodesic X-ray transform. Since $(M,g)$ is simple, Theorem \ref{thm_xray} gives that $q_1=q_2$.
\end{proof}

\begin{remark}
It is natural that one needs $T$ to be sufficiently large in Theorem \ref{thm_gelfand_uniqueness}. By finite propagation speed the map $\Lambda_q$ is unaffected if one changes $q$ outside the set $\{ x \in M \,:\, \mathrm{dist}(x,\p M) < T/2 \}$.\footnote{If $u$ and $\tilde{u}$ solve \eqref{wavedp} for potentials $q$ and $\tilde{q}$ with the same Dirichlet data $f$, and if $q = \tilde{q}$ in $U := \{ x \in M  \,:\, \mathrm{dist}(x,\p M) < T/2 \}$, then $w := u - \tilde{u}$ solves $(\Box + q)w = F$ where $F := -(q-\tilde{q})\tilde{u}$ vanishes in $U \times (0,T)$ and also in $(M \setminus U) \times (0,T/2)$ since $\tilde{u}$ vanishes there. Moreover, $w = \p_t w = 0$ on $\{ t=0 \}$ and $w|_{\p M \times (0,T)} = 0$. By finite speed of propagation $\p_{\nu} w|_{\p M \times (0,T)} = 0$. This proves that $\Lambda_q = \Lambda_{\tilde{q}}$.}
\end{remark}

Theorem \ref{thm_gelfand_xray} could be proved based on the following facts, see e.g.\ \cite{StefanovYang}:
\begin{enumerate}
\item[1.]
The map $\Lambda_q$ is an FIO of order $1$ on $\p M \times (0,T)$.
\item[2.]
The X-ray transform of $q$ can be read off from the principal symbol of $\Lambda_q - \Lambda_0$.
\end{enumerate}
We will give a direct proof that avoids the first step and is based on testing $\Lambda_q$ against highly oscillatory boundary data. This follows the idea that the principal symbol of a $\Psi$DO (or FIO) can be obtained by testing against oscillatory functions, e.g.\ 
\begin{equation} \label{principal_symbol_oscillatory}
\sigma_{\mathrm{pr}}(A)(x_0,\xi_0) = \lim_{\lambda \to \infty} \lambda^{-m} e^{-i\lambda x \cdot \xi_0} A(e^{i \lambda x \cdot \xi_0}) \Big|_{x=x_0}
\end{equation}
when $A$ is a classical $\Psi$DO of order $m$ in $\mR^n$. Our proof of Theorem \ref{thm_gelfand_xray} is based on \emph{geometric optics solutions} and will be done in two parts. First we assume that $(M,g)$ is simple and use a classical geometric optics construction with real phase function. In the general case we employ a Gaussian beam construction with complex phase function.

Theorem \ref{thm_gelfand_uniqueness} is in fact true even for a general compact manifold $(M,g)$ under the sharp condition $T > 2 \sup_{x \in M} \mathrm{dist}(x, \p M)$. This and many other results for \emph{time-independent} coefficients follow from the \emph{Boundary Control method} introduced in \cite{Belishev}, see \cite{KKL, Lassas} for further developments. However, the Boundary Control method is not in general available when $q = q(x,t)$ is \emph{time-dependent}. This case arises in inverse problems for nonlinear equations or in general relativity. In that case (and if one considers the analogous problem on $\p M \times \mR$ instead of $\p M \times (0,T)$, see Exercise \ref{ex_wave}), the geometric optics method still works and gives that 
\begin{equation} \label{light_ray}
\int_0^{\ell} q_1(\gamma(t), t+\sigma) \,dt = \int_0^{\ell} q_2(\gamma(t), t+\sigma) \,dt
\end{equation}
whenever $\gamma$ is a maximal geodesic as above and $\sigma \in \mR$ is a time-delay parameter. This means that the \emph{light ray transforms} of $q_1$ and $q_2$ are the same. The curves $(\gamma(t), t+\sigma)$ where $\gamma$ is a geodesic in $M$ are called \emph{light rays}; they are lightlike, or null, geodesics for the \emph{Lorentzian metric} $-dt^2 + g(x)$. When $(M,g)$ is simple the invertibility of the light ray transform follows from invertibility of the geodesic X-ray transform, see Exercise \ref{ex_wave}.

More generally, instead of the wave operator $\Box = \p_t^2 - \Delta$ corresponding to the product Lorentzian metric $-dt^2 + g(x)$ in $M \times \mR$, one could consider a more general Lorentzian metric $\bar{g}$ (i.e.\ a symmetric $2$-tensor field on $M \times \mR$ that has one negative and $n$ positive eigenvalues at each point) and the corresponding wave operator $\Box_{\bar{g}}$. Inverse problems for $\Box_{\bar{g}}$ constitute a wave equation analogue of the anisotropic Calder\'on problem (see Section \ref{sec_calderon}).

The following questions remain open:

\begin{question}
Can one recover a time-dependent potential $q \in C^{\infty}_c(M \times \mR)$ from the hyperbolic DN map on $\p M \times \mR$ for a general compact Riemannian manifold $(M,g)$ with boundary?
\end{question}

\begin{question}
For which Lorentzian metrics $\bar{g}$ is the light ray transform invertible?
\end{question}

\begin{question}
For which Lorentzian metrics $\bar{g}$ does one have uniqueness in the Gelfand problem?
\end{question}

See \cite{AFO, FIKO, FIO, Stefanov} for recent results on the above questions. We also mention that for \emph{nonlinear} wave equations better results are available, see e.g.\ \cite{Lassas}.

We now start the proof of Theorem \ref{thm_gelfand_xray}. Alternative presentations may be found in the lecture notes \cite{Oksanen_lectures, Salo_lectures} in the case of simple manifolds or Euclidean space. Similar results in much more general settings appear in \cite{StefanovYang, OSSU}. The proof proceeds in three steps.
\begin{enumerate}
\item[1.] 
Derivation of an integral identity showing that if $\Lambda_{q_1}=\Lambda_{q_2}$, then $q_1-q_2$ is $L^2$-orthogonal to certain products of solutions.
\item[2.] 
Construction of special solutions that concentrate near a light ray $(\gamma(t), t + \sigma)$ for some $\sigma > 0$.
\item[3.] 
Proof of \eqref{wave_xray} by inserting the special solutions in the integral identity and taking a limit.
\end{enumerate}

\subsection{Integral identity}

\begin{lemma}[Integral identity] \label{lemma_gelfand_integral_identity}
Assume that $q_1, q_2 \in C^{\infty}(M)$. For any $f_1, f_2 \in C^{\infty}_c(\p M \times (0,T))$, one has  
\[
((\Lambda_{q_1} - \Lambda_{q_2}) f_1, f_2)_{L^2(\p M \times (0,T))} = \int_{M} \int_0^T (q_1-q_2) u_1 \bar{u}_2 \,dt \,dV
\]
where $u_1$ solves \eqref{wavedp} with $q=q_1$ and $f=f_1$, and $u_2$ solves an analogous problem with vanishing Cauchy data on $\{ t=T \}$:
\begin{equation} \label{wavedptwo}
\left\{ \begin{array}{rll}
(\Box + q_2) u_2 &\!\!\!= 0 & \quad \text{in } M \times (0,T), \\
u_2 &\!\!\!= f_2 & \quad \text{on } \partial M \times (0,T), \\
u_2 = \p_t u_2 &\!\!\!= 0 & \quad \text{on } \{ t = T \}.
\end{array} \right.
\end{equation}
\end{lemma}
\begin{proof}
We first compute the formal adjoint of the DN map: one has 
\[
( \Lambda_q f, h)_{L^2(\p M \times (0,T))} = ( f, \Lambda_q^T h)_{L^2(\p M \times (0,T))}
\]
where $\Lambda_q^T h = \p_{\nu} v|_{\p M \times (0,T)}$ with $v$ solving $(\Box + q)v = 0$ so that $v|_{\p M \times (0,T)} = h$ and $v = \p_t v = 0$ on $\{ t = T \}$. To prove this, we let $u$ be the solution of \eqref{wavedp} and integrate by parts:
\begin{align*}
( \Lambda_q f, h)_{L^2(\p M \times (0,T))} &= \int_{\p M} \int_0^T (\p_{\nu} u) \bar{v} \,dt \,dS \\
 &= \int_{M} \int_0^T (\langle \nabla u, \nabla \bar{v} \rangle + (\Delta u) \bar{v}) \,dt \,dV \\
 &= \int_{M} \int_0^T (\langle \nabla u, \nabla \bar{v} \rangle + (\p_t^2 u + qu) \bar{v}) \,dt \,dV \\
 &= \int_{M} \int_0^T (\langle \nabla u, \nabla \bar{v} \rangle - \p_t u \p_t \bar{v} + qu \bar{v}) \,dt \,dV \\
 &= \int_{M} \int_0^T (\langle \nabla u, \nabla \bar{v} \rangle + u (\ol{\p_t^2 v + qv}) ) \,dt \,dV \\
 &= \int_{M} \int_0^T (\langle \nabla u, \nabla \bar{v} \rangle + u \Delta \ol{v}) \,dt \,dV \\
 &= \int_{\p M} \int_0^T u \p_{\nu} \bar{v} \,dt \,dS \\
 &= ( f, \Lambda_q^T h)_{L^2(\p M \times (0,T))}.
\end{align*}

Now, if $u_1$ and $u_2$ are as stated, the first half of the computation above gives 
\begin{align*}
( \Lambda_{q_1} f_1, f_2)_{L^2(\p M \times (0,T))} 
 &= \int_{M} \int_0^T (\langle \nabla u_1, \nabla \bar{u}_2 \rangle- \p_t u_1 \p_t \bar{u}_2 + q_1 u_1 \bar{u}_2) \,dt \,dV \\
 \intertext{and the second half of the computation above gives} 
( \Lambda_{q_2} f_1, f_2)_{L^2(\p M \times (0,T))} 
 &= ( f_1, \Lambda_{q_2}^T f_2)_{L^2(\p M \times (0,T))} \\
 &= \int_{\Omega} \int_0^T (\langle \nabla u_1, \nabla \bar{u}_2 \rangle- \p_t u_1 \p_t \bar{u}_2 + q_2 u_1 \bar{u}_2) \,dt \,dV. 
\end{align*}
The result follows by subtracting these two identities.
\end{proof}

If $\Lambda_{q_1} = \Lambda_{q_2}$, it follows from Lemma \ref{lemma_gelfand_integral_identity} that 
\[
\int_{M} \int_0^T (q_1-q_2) u_1 \bar{u}_2 \,dt \,dV = 0
\]
for all solutions $u_1$ and $u_2$ of the given type.

\subsection{Recovering the X-ray transform}

We will now start the construction of special solutions concentrating near a light ray $(\gamma(t), t + \sigma)$ where $\sigma > 0$ is a small time delay parameter. We use the method of \emph{geometrical optics}, also known as the \emph{WKB method}, and first look for approximate solutions using the ansatz 
\[
v(x,t) = e^{i\lambda \varphi(x,t)} a(x,t)
\]
where $\lambda > 0$ is a large parameter, $\varphi$ is a real phase function, and $a$ is an amplitude supported near the curve $t \mapsto (\gamma(t), t+\sigma)$. (We could also write $h = 1/\lambda$ and state our results in terms of a small parameter $h$, but this is not necessary since our arguments are elementary and we do need any semiclassical $\Psi$DO calculus. Using a large parameter may be less confusing since then we do not need to multiply operators by powers of $h$.)

A direct computation, given below in \eqref{wave_equation_geometric_optics}, shows that 
\begin{equation} \label{wave_eq_conjugation_first}
(\Box+q)v = e^{i\lambda \varphi} \big[ \lambda^2 \left[ \abs{\nabla_x \varphi}_g^2 - (\p_t \varphi)^2 \right] a + i \lambda La + (\Box + q)a \big]
\end{equation}
where $L$ is a certain first order differential operator. Now $v$ is a good approximate solution if the right hand side is very small when $\lambda$ is large. In particular, we want the $\lambda^2$ term to vanish, which means that the phase function $\varphi$ should solve the \emph{eikonal equation} 
\begin{equation} \label{eikonal_eq}
\abs{\nabla_x \varphi}_g^2 - (\p_t \varphi)^2 = 0.
\end{equation}
We will show that when $(M,g)$ is simple, the function $\varphi(x,t) := t - r$ is a solution where $(\omega, r)$ are Riemannian polar coordinates as in formula \eqref{riemannian_polar_first}. Here and below it is convenient to write $(\omega,r)$ instead of $(r,\omega)$. We also show that by solving transport equations involving $L$ one can obtain an amplitude $a$ supported near the curve $t \mapsto (\gamma(t), t+\sigma)$ satisfying 
\[
\norm{i \lambda La + (\Box + q)a}_{L^{\infty}} \to 0 \text{ as $\lambda \to \infty$}.
\]
Thus $v$ is an approximate solution in the sense that $(\Box+q)v = o(1)$ as $\lambda \to \infty$.
These approximate solutions can then be converted into exact solutions by solving a Dirichlet problem for the wave equation.

After the outline above, we give the precise statement regarding concentrating solutions.

\begin{proposition}[Concentrating solutions] \label{prop_gelfand_concentrating_solutions}
Assume that $q \in C^{\infty}_c(M^{\mathrm{int}})$, and let $\gamma: [0,\ell] \to M$ be a maximal geodesic in $M$ with $\ell < T$. Let also $\sigma > 0$ be a small enough time delay parameter.
For any $\lambda \geq 1$ there is a solution $u = u_{\lambda}$ of $(\Box + q) u = 0$ in $M \times (0,T)$ with $u = \p_t u = 0$ on $\{ t=0 \}$, such that for any $\psi \in C^{\infty}_c(M \times [0,T])$ one has  
\begin{equation} \label{concentrating_solutions_first_limit}
\lim_{\lambda \to \infty} \int_{M} \int_0^T \psi \abs{u}^2 \,dt \,dV = \int_{0}^{\ell} \psi(\gamma(t), t + \sigma) \,dt.
\end{equation}
Moreover, if $\tilde{q} \in C^{\infty}_c(M^{\mathrm{int}})$, there is a solution $\tilde{u} = \tilde{u}_{\lambda}$ of $(\Box + \tilde{q})\tilde{u} = 0$ in $M \times (0,T)$ with $\tilde{u} = \p_t \tilde{u} = 0$ on $\{ t=T \}$, such that for any $\psi \in C^{\infty}_c(M \times [0,T])$ one has  
\begin{equation} \label{concentrating_solutions_second_limit}
\lim_{\lambda \to \infty} \int_{M} \int_0^T \psi u \ol{\tilde{u}} \,dt \,dV  = \int_{0}^{\ell} \psi(\gamma(t), t + \sigma) \,dt.
\end{equation}
\end{proposition}

\begin{remark} \label{rem_wave_sing}
The fact that one can construct solutions to the wave equation that concentrate near light rays $t \mapsto (\gamma(t), t + \sigma)$ is a consequence of \emph{propagation of singularities}. This general phenomenon states that singularities of solutions for operators with real valued principal symbol $p$ propagate along \emph{null bicharacteristic curves}, i.e.\ integral curves of the Hamilton vector field $H_p$ in phase space. The principal symbol of the wave operator $\Box$ is $p(x,t,\xi,\tau) = -\tau^2 + \abs{\xi}_g^2$, and the light rays are projections to the $(x,t)$ variables of null bicharacteristic curves for $\Box$. This point of view will be emphasized in the context of Gaussian beams in Section \ref{subseq_gaussian_beam} below.

We also mention that \eqref{concentrating_solutions_first_limit} indicates that the \emph{semiclassical measure}, or \emph{quantum limit}, of the family $(u_{\lambda})$ as $\lambda \to \infty$ is the delta function of the light ray. One could also give more precise phase space versions of the statements in Proposition \ref{prop_gelfand_concentrating_solutions} (see \cite{OSSU}).
\end{remark}

\begin{proof}[Proof of Theorem \ref{thm_gelfand_xray}]
Using the assumption $\Lambda_{q_1} = \Lambda_{q_2}$ and Lemma \ref{lemma_gelfand_integral_identity}, we have 
\begin{equation} \label{gelfand_orthogonality_relation}
\int_{M} \int_0^T (q_1-q_2) u_1 \ol{u}_2 \,dt \,dV = 0
\end{equation}
for any solutions $u_j$ of $(\Box + q_j) u_j = 0$ in $M \times (0,T)$ so that $u_1 = \p_t u_1 = 0$ on $\{ t = 0 \}$, and $u_2 = \p_t u_2 = 0$ on $\{ t = T \}$.

Let $\gamma: [0,\ell] \to M$ be a maximal unit speed geodesic segment in $M$ with $\ell < T$, let $\sigma > 0$ be small, and let $u_1 = u_{1,\lambda}$ be the solution constructed in Proposition \ref{prop_gelfand_concentrating_solutions} for the potential $q_1$ with $u_1 = \p_t u_1 = 0$ on $\{ t = 0 \}$. Moreover, let $u_2 = u_{2,\lambda}$ be the solution constructed in the end of Proposition \ref{prop_gelfand_concentrating_solutions} for the potential $q_2$ with $u_2 = \p_t u_2 = 0$ on $\{ t = T \}$. Taking the limit as $\lambda \to \infty$ in \eqref{gelfand_orthogonality_relation} and using \eqref{concentrating_solutions_second_limit} with $\psi(x,t) = (q_1-q_2)(x)$, we obtain that 
\[
\int_{0}^{\ell} (q_1-q_2)(\gamma(t)) \,dt = 0.
\]
This is true for any maximal geodesic $\gamma$ in $M$ with length $\ell < T$, which proves the result. 
\end{proof}

It remains to prove Proposition \ref{prop_gelfand_concentrating_solutions}. We first give a proof under the additional assumption that $(M,g)$ is \emph{simple}. In this case it is possible to solve the eikonal equation globally and a standard geometric optics construction is sufficient. For general manifolds $(M,g)$ there may be conjugate points and it may not be possible to find smooth global solutions of the eikonal equation. We will use a Gaussian beam construction to deal with this case.

\subsection{Special solutions in the simple case -- geometrical optics} \label{sec_wave_simple}

Recall that we are looking for approximate solutions of the form $v = e^{i\lambda \varphi} a$. For the construction of the phase function $\varphi$ we will use Lemma \ref{lemma_simple_exp}, which essentially states that a manifold is simple iff it admits global polar coordinates $(\omega, r)$ centered at any point. We will need also need the following property.

\begin{lemma}[Riemannian polar coordinates] \label{lemma_polar}
In the $(\omega,r)$ coordinates the metric has the form 
\[
g(\omega,r) = \left( \begin{array}{cc}  g_0(\omega,r) & 0 \\ 0 & 1 \end{array} \right).
\]
\end{lemma}
\begin{proof}
It is enough to prove that $\langle \p_r, \p_r \rangle = 1$ and $\langle \p_r, w \rangle = 0$, where $w = \dot{\eta}(0)$ for any curve $\eta(t) = (r, \omega(t))$. Since $\p_r$ is the tangent vector of a unit speed geodesic starting at $p$, one has $\langle \p_r, \p_r \rangle = 1$. If $\eta(t)$ is a curve as above, the fact that $\langle \p_r, w \rangle = 0$ is precisely the content of the Gauss lemma in Riemannian geometry (see e.g.\ \cite[Section 3.7]{PSU_book}).
\end{proof}

We can now prove the result on concentrating solutions. The proof is elementary although a bit long.

\begin{proof}[Proof of Proposition \ref{prop_gelfand_concentrating_solutions} when $(M,g)$ is simple]
Let $\gamma: [0,\ell] \to M$ be a maximal unit speed geodesic in $M$ with $\ell < T$, and let initially $\sigma \in (0,T-\ell)$.

We first construct an approximate solution $v = v_{\lambda}$ for the operator $\Box + q$, having the form 
\[
v(x,t) = e^{i\lambda \varphi(x,t)} a(x,t)
\]
where $\varphi$ is a real phase function, and $a$ is an amplitude supported near the curve $t \mapsto (\gamma(t), t + \sigma)$. Note that 
\begin{align*}
\p_t(e^{i\lambda \varphi} u) &= e^{i\lambda \varphi}(\p_t + i \lambda \p_t \varphi) u, \\
\p_t^2(e^{i\lambda \varphi} u) &= e^{i\lambda \varphi}(\p_t + i \lambda \p_t \varphi)^2 u
\end{align*}
and similarly for the $x$-derivatives 
\begin{align*}
\nabla(e^{i\lambda \varphi} u) &= e^{i\lambda \varphi}(\nabla + i \lambda \nabla \varphi) u, \\
\mathrm{div} \nabla(e^{i\lambda \varphi} u) &= e^{i\lambda \varphi}(\mathrm{div} + i \lambda \langle \nabla \varphi, \,\cdot\, \rangle)(\nabla + i \lambda \nabla \varphi) u.
\end{align*}
We thus compute 
\begin{align}
(\Box+q)(e^{i\lambda \varphi} a) &= e^{i\lambda \varphi}( (\p_t + i\lambda \p_t \varphi)^2 - (\mathrm{div}_x + i \lambda \langle \nabla_x \varphi, \,\cdot\, \rangle)(\nabla_x + i \lambda \nabla_x \varphi) + q) a \notag \\
 &= e^{i\lambda \varphi} \big[ \lambda^2 \left[ \abs{\nabla_x \varphi}_g^2 - (\p_t \varphi)^2 \right] a \notag \\
 & \qquad + i \lambda \left[2 \p_t \varphi \p_t a - 2 \langle \nabla_x \varphi, \nabla_x a \rangle + (\Box \varphi)a \right] + (\Box + q)a \big]. \label{wave_equation_geometric_optics}
\end{align}

We would like to have $(\Box+q)(e^{i\lambda \varphi} a) = O(\lambda^{-1})$, so that $v = e^{i\lambda \varphi} a$ would indeed be an approximate solution when $\lambda$ is large. To this end, we first choose $\varphi$ so that the $\lambda^2$ term in \eqref{wave_equation_geometric_optics} vanishes. This will be true if $\varphi$ solves the \emph{eikonal equation} 
\begin{equation} \label{eikonal1}
\abs{\nabla_x \varphi}_g^2 - (\p_t \varphi)^2 = 0.
\end{equation}
We make the simple choice 
\begin{equation} \label{phase_choice}
\varphi(x,t) := t - \psi(x)
\end{equation}
where $\psi \in C^{\infty}(M)$ should solve the equation 
\begin{equation} \label{eikonal2}
\abs{\nabla \psi}_g^2 = 1.
\end{equation}
This is another eikonal equation, now only in the $x$ variables. We now invoke the assumption that $(M,g)$ is \emph{simple} and give an explicit solution of \eqref{eikonal2}. 

Let $(U,g)$ be an open manifold as in Lemma \ref{lemma_simple_exp} that contains $M$ as a compact subdomain. Let $\eta$ be the maximal geodesic in $U$ with $\eta|_{[0,\ell]} = \gamma$ and, possibly after decreasing $\sigma > 0$, $p := \eta(-\sigma) \in U \setminus M$. By Lemma \ref{lemma_simple_exp}, if $D_p$ is the maximal domain of $\exp_p$ in $T_p U$, then 
\[
\exp_p: D_p \to U
\]
is a diffeomorphism. Thus any point $x \in U$ can be written uniquely as 
\[
x = \exp_p(r\omega)
\]
for some $r \geq 0$ and $\omega \in S^{n-1}$ with $r \omega \in D_p$. Identifying $x$ with $(\omega,r)$ gives global coordinates in $U \setminus \{ p \}$. We claim that 
\[
\psi(\omega,r) := r
\]
is a smooth solution of \eqref{eikonal2} near $M$. Note first that $\psi$ is smooth in $M$, since the origin of polar coordinates is outside $M$. Now the fact that $\psi$ solves \eqref{eikonal2} follows immediately from Lemma \ref{lemma_polar} since 
\[
\langle \nabla \psi, \nabla \psi \rangle = \langle \p_r, \p_r \rangle = 1.
\]

With the choice $\varphi(x,t) = t - \psi(x)$, we have \eqref{eikonal1} and thus the equation \eqref{wave_equation_geometric_optics} becomes 
\begin{equation} \label{wave_equation_geometric_optics_two}
(\Box+q)(e^{i\lambda \varphi} a) = e^{i\lambda \varphi} \left[ i\lambda (La) + (\Box+q)a \right]
\end{equation}
where $L$ is the first order operator defined by 
\[
La := 2 \p_t \varphi \p_t a - 2 \langle \nabla_x \varphi, \nabla_x a \rangle + (\Box \varphi)a.
\]
Now $\p_t \varphi = 1$, and since $\psi(\omega, r) = r$ we obtain from Lemma \ref{lemma_polar} that 
\[
\langle \nabla_x \varphi, \nabla_x a \rangle = g^{jk} \p_{x_j} \varphi \p_{x_k} a = -\p_r a.
\]
Writing $b := \Box \varphi$, the operator $L$ simplifies to 
\begin{equation} \label{l_second}
La = 2 (\p_t + \p_r)a + ba.
\end{equation}
For later purposes we observe that by Lemma \ref{lemma_polar}, $b$ has the precise form 
\begin{align}
b(\omega,r,t) &= (\p_t^2 - \Delta_x) \varphi = \Delta_x r = \det(g)^{-1/2}\p_r (\det(g)^{1/2}) \label{b_precise_form} \\
 &= \frac{1}{2} \p_r \left[ \log \det(g(\omega,r))  \right]. \notag
\end{align}

We next look for the amplitude $a$ in the form 
\[
a = a_0 + \lambda^{-1} a_{-1}.
\]
Inserting this to \eqref{wave_equation_geometric_optics} and equating like powers of $\lambda$, we get 
\begin{multline} \label{wave_equation_geometric_optics_three}
(\Box+q)(e^{i\lambda \varphi} a) \\
 = e^{i\lambda \varphi} \left[ i\lambda (La_0) + \left[ i L a_{-1} + (\Box+q)a_0 \right] + \lambda^{-1} (\Box+q)a_{-1} \right].
\end{multline}
We would like the last expression to be $O(\lambda^{-1})$. This will hold if $a_0$ and $a_{-1}$ satisfy the \emph{transport equations} 
\begin{align} \label{wave_transport_equations}
\left\{ \begin{array}{rl}
L a_0 &\!\!\!= 0, \\[3pt]
La_{-1} &\!\!\!= i(\Box+q)a_0.
\end{array} \right.
\end{align}
It is not hard to solve these transport equations. To do this, it is convenient to switch from the coordinates $(\omega,r,t)$ near $M \times (0,T)$ to new coordinates $(\omega, z, w)$, where 
\begin{equation} \label{new_z_w_coordinates}
z = \frac{t+r}{2}, \qquad w = \frac{t-r}{2}.
\end{equation}
Then $L$ in \eqref{l_second} simplifies to $2 \p_z + b$ in the sense that 
\[
LF(\omega,r,t) = (2 \p_z \breve{F} + \breve{b} \breve{F})(\omega,\frac{t+r}{2},\frac{t-r}{2})
\]
where $\breve{F}$ corresponds to $F$ in the new coordinates:
\[
\breve{F}(\omega,z,w) := F(\omega,z-w,z+w).
\]
Finally, we can use an integrating factor to get rid of $\breve{b}$. One has 
\begin{equation} \label{l_integrating_factor}
LF(\omega,r,t) = 2 c^{-1} \p_z (c \breve{F})(\omega,\frac{t+r}{2},\frac{t-r}{2})
\end{equation}
provided that $2 c^{-1} \p_z c = \breve{b}$. By \eqref{b_precise_form}, this will hold if we choose $c$ so that 
\begin{equation} \label{c_integrating_factor}
c(\omega, z, w) := \det(g(\omega,z-w))^{1/4}.
\end{equation}

We can now solve the transport equations \eqref{wave_transport_equations}. By \eqref{l_integrating_factor} the first transport equation reduces to 
\[
\p_z (c \breve{a}_0) = 0.
\]
Recall that we want our amplitude $a$ to be supported near the curve $t \mapsto (\eta(t), t + \sigma)$ in the $(x,t)$ coordinates. Recall also that the center $p$ of our polar coordinates was given by $p = \eta(-\sigma)$. Thus $\eta(t) = (\omega_0, t+\sigma)$ for some $\omega_0 \in S^{n-1}$ in the $(\omega, r)$ coordinates, and at time $\sigma+t$ the amplitude should be supported near $(\omega_0, \sigma+t)$. Because of these facts, it makes sense to choose 
\[
\breve{a}_0(\omega,z,w) := c(\omega,z,w)^{-1} \chi(\omega,w),
\]
where $\chi \in C^{\infty}_c(S^{n-1} \times \mR)$ is supported near $(\omega_0, 0)$. We will later choose $\chi$ to depend on $\lambda$. Note also that $\gamma(t)$ exits $M$ when $t=\ell$, which means that 
\[
\breve{a}_0|_{M \times [\sigma+\ell+\eps, \sigma+\ell+2\eps]} = 0
\]
for some $\eps > 0$ if $\sigma$ is chosen so small that $\sigma + \ell < T$. We set $\breve{a}_0 = 0$ for $t \in [\sigma+\ell+\eps, T]$.

Next we choose 
\[
\breve{a}_{-1}(\omega,z,w) := -\frac{1}{2i c} \int_0^z c ((\Box + q)a_0){\,\breve{\rule{0pt}{6pt}}\,}(\omega,s,w) \,ds.
\]
The functions $a_0$ and $a_{-1}$ satisfy \eqref{wave_transport_equations}, and they vanish unless $w$ is small (i.e.\ $r$ is close to $t$). Then \eqref{wave_equation_geometric_optics_three} becomes 
\[
(\Box+q)(e^{i\lambda \varphi} a) = F_{\lambda}
\]
where 
\[
F_{\lambda} := \lambda^{-1} e^{i\lambda \varphi} (\Box+q)a_{-1}.
\]
Using the Cauchy-Schwarz inequality, one can check that 
\begin{align}
\norm{F_{\lambda}}_{L^{\infty}(M \times (0,T))} &\leq \lambda^{-1} \norm{(\Box + q) a_{-1}}_{L^{\infty}(M \times (0,T))} \notag \\
 &\leq C \lambda^{-1} \norm{\chi}_{W^{4,\infty}(S^{n-1} \times \mR)} \label{flambda_estimate}
\end{align}
uniformly over $\lambda \geq 1$. This concludes the construction of the approximate solution $v = e^{i\lambda \varphi} a$.

We next find an exact solution $u = u_{\lambda}$ of \eqref{wavedp} having the form 
\[
u = v + R
\]
where $R$ is a correction term. Note that for $t$ close to $0$, $v(\,\cdot\,,t)$ is supported near $p \notin M$ and hence $v = \p_t v = 0$ on $M \times \{t=0\}$. Note also that $(\Box + q)v = F_{\lambda}$. Thus $u$ will solve \eqref{wavedp} for $f = v|_{\p M \times (0,T)}$ if $R$ solves 
\begin{equation} \label{wavedp_correction_term}
\left\{ \begin{array}{rll}
(\Box + q) R &\!\!\!= -F_{\lambda} & \quad \text{in } M \times (0,T), \\
R &\!\!\!= 0 & \quad \text{on } \partial M \times (0,T), \\
R = \p_t R &\!\!\!= 0 & \quad \text{on } \{ t = 0 \}.
\end{array} \right.
\end{equation}
By the well-posedness of this problem (see \cite[Theorem 5 in \S 7.2.3]{Evans} for the Euclidean case, again the proof in the Riemannian case is the same), there is a unique solution $R$ with 
\begin{equation} \label{r_estimates}
\norm{R}_{L^{\infty}((0,T) ; H^1(M))} \leq C \norm{F_{\lambda}}_{L^2((0,T) ; L^2(M))} \leq C \lambda^{-1} \norm{\chi}_{W^{4,\infty}}.
\end{equation}

We now fix the choice of $\chi$ so that \eqref{concentrating_solutions_first_limit} will hold. Recall that $\chi \in C^{\infty}_c(S^{n-1} \times \mR)$ is supported near $(\omega_0, 0)$. We may parametrize a neighborhood of $\omega_0$ in $S^{n-1}$ by points $y' \in \mR^{n-1}$ so that $\omega_0$ corresponds to $0$, and thus we may think of $\chi$ as a function in $\mR^n$ supported near $0$. Let $\zeta \in C^{\infty}_c(\mR^n)$ satisfy $\zeta = 1$ near $0$ and $\norm{\zeta}_{L^2(\mR^n)} = 1$, and choose 
\[
\chi(y) := \eps^{-n/2} \zeta(y/\eps)
\]
where 
\[
\eps = \eps(\lambda) = \lambda^{-\frac{1}{n+8}}.
\]
With this choice 
\[
\norm{\chi}_{L^2(\mR^n)} = 1, \qquad \norm{\chi}_{W^{4,\infty}(\mR^n)} \lesssim \eps^{-n/2-4} \lesssim \lambda^{1/2}.
\]
It follows from \eqref{r_estimates} that 
\[
\norm{v}_{L^2(M \times (0,T))} \lesssim 1, \qquad \norm{R}_{L^2(M \times (0,T))} \lesssim \lambda^{-1/2}.
\]
Since $u = v + R$, the integral in \eqref{concentrating_solutions_first_limit} has the form 
\begin{align*}
\int_{M} \int_0^T \psi \abs{u}^2 \,dV \,dt &= \int_{M} \int_0^T \psi \abs{v}^2 \,dV \,dt + O(\lambda^{-1/2}) \\
 &= \int_{M} \int_0^T \psi \abs{a_0}^2 \,dV \,dt + O(\lambda^{-1/2}).
\end{align*}
We recall that $\breve{a}_0 = c^{-1} \chi$ where $c$ is given by \eqref{c_integrating_factor}. Using that $\psi \abs{a_0}^2$ is compactly supported in $M^{\mathrm{int}} \times (0,T)$, changing variables according to \eqref{new_z_w_coordinates}, and identifying $y' \in \mR^{n-1}$ with $\omega \in S^{n-1}$, we obtain 
\begin{align*}
 &\int_{M} \int_0^T \psi \abs{a_0}^2 \,dV \,dt = \int_{S^{n-1}} \int \int (\psi \abs{a_0}^2 \det(g)^{1/2})(\omega,r,t) \,d\omega \,dr \,dt \\
 & \qquad = \int_{\mR^{n+1}} \psi(y',z-w,z+w) \chi(\omega,w)^2 \frac{\det(g(y',z-w))^{1/2} }{c(y',z,w)^{2}} \,dy' \,dz \,dw.
 \end{align*}
By \eqref{c_integrating_factor} one has $\frac{\det(g(y',z-w))^{1/2} }{c(y',z,w)^{2}} \equiv 1$. It follows that 

\begin{align*}
 &\int_{M} \int_0^T \psi \abs{u}^2 \,dV \,dt \\
 & \qquad = \int_{\mR^{n+1}} \psi(y',z-w,z+w) \eps^{-n} \zeta(y'/\eps, w/\eps)^2 \,dy' \,dz \,dw + O(\lambda^{-1/2}).
\end{align*}
Finally, changing $y'$ to $\eps y'$ and $w$ to $\eps w$ and letting $\lambda \to \infty$ (so $\eps \to 0$) yields 
\begin{align*}
\lim_{\lambda \to \infty} \int_{M} \int_0^T \psi \abs{u}^2 \,dV \,dt &= \int_{\mR^{n+1}} \psi(0',z,z) \zeta(y',w)^2 \,dy' \,dz \,dw \\
 &= \int_{-\infty}^{\infty} \psi(0',z,z) \,dz
\end{align*}
by the normalization $\norm{\zeta}_{L^2(\mR^n)} = 1$ and the fact that $\psi \in C^{\infty}_c(M^{\mathrm{int}} \times [0,T])$. Undoing the changes of coordinates, we see that the curve $(0',z,z)$ in the $(y',r,t)$ coordinates corresponds to $t \mapsto (\omega_0,t,t)$ in the $(\omega, r, t)$ coordinates. Thus 
\[
\int_{-\infty}^{\infty} \psi(0',z,z) \,dz = \int_0^{\ell} \psi(\gamma(t),t+\sigma) \,dt
\]
which proves \eqref{concentrating_solutions_first_limit}.

It remains to prove \eqref{concentrating_solutions_second_limit}. It is enough to construct a solution $\tilde{u} = v + \tilde{R}$ for potential $\tilde{q}$ where $\tilde{R} = \p_t \tilde{R} = 0$ on $\{ t=T \}$. Since $\gamma(t)$ exits $M$ after time $\ell < T$, we have $v|_{M \times [\sigma+\ell+\eps, \sigma+\ell+2\eps]} = 0$ for some small $\eps > 0$. Redefining $v$ to be zero for $t \geq \sigma+\ell+2\eps$, we see that \eqref{flambda_estimate} still holds. Then we choose $\tilde{R}$ solving \eqref{wavedp_correction_term} but with $\tilde{R} = \p_t \tilde{R} = 0$ on $\{ t=T \}$ instead of $\{ t=0 \}$. We can do such a construction for the potential $\tilde{q}$ instead of $q$. Since $\varphi$ and $a_0$ are independent of the potential $q$, the same argument as above proves \eqref{concentrating_solutions_second_limit}.
\end{proof}

\subsection{Special solutions in the general case -- Gaussian beams} \label{subseq_gaussian_beam}

Above we used the assumption that $(M,g)$ is simple in order to find a global smooth real valued solution $\varphi$ of the eikonal equation in $M$. For general manifolds $(M,g)$ such solutions do not exist in general. We will remedy this by using a \emph{Gaussian beam} construction, which involves two modifications:
\begin{enumerate}
\item[1.]
The phase function is \emph{complex valued} with imaginary part growing quadratically away from the curve of interest.
\item[2.]
The eikonal equation is only solved to \emph{infinite order on the curve}, instead of globally. This boils down to finding a solution of a \emph{matrix Riccati equation} with positive definite imaginary part.
\end{enumerate}
The Gaussian beam construction is classical and goes back to \cite{BabichLazutkin, Hormander1971}. We will sketch briefly the main ideas in this construction and we refer to \cite{KKL, Ralston, Ralston_note} for further details. A general version of this construction, with applications to inverse problems, is given in \cite{OSSU}.

\subsubsection{General setup}

Let $(M,g)$ be a general compact manifold with boundary, embedded in a closed manifold $(N,g)$. Recall that we would like to find an approximate solution of $Pv = 0$ in $M \times [0,T]$ where $P = \p_t^2 - \Delta_g$, so that $v$ concentrates near a light curve $(\eta(t), t + \sigma)$ where $\eta(t)$ is a geodesic in $M$ and $\sigma$ is a time delay parameter.

It turns out that the construction of concentrating solutions depends very little on the fact that one is working with the wave equation. In fact the Gaussian beam construction can be carried out for any differential operator $P$ on a manifold $X$ and one can find approximate solutions concentrating near a curve $x: [0,T] \to X$, provided that 
\begin{enumerate}
\item[(a)]  
the principal symbol $p \in C^{\infty}(T^* X)$ of $P$ is \emph{real valued}; and 
\item[(b)]  
the curve $x(t)$ is the spatial projection of a \emph{null bicharacteristic curve} $\gamma: [0,T] \to T^* X$, i.e.\ $x(t) = \pi(\gamma(t))$ where $\pi: T^* X \to X$ is the base projection to $X$. We will also assume that 
\begin{itemize}
\item 
$\gamma: [0,T] \to T^* X$ is injective (i.e.\ non-periodic, so that we do not need to worry about the curve closing); and 
\item 
$\dot{x}(t) \neq 0$ for $t \in [0,T]$ (i.e.\ $x(t)$ has no cusps, to simplify the construction).
\end{itemize}
\end{enumerate}
The principal symbol of $\p_t^2 - \Delta_g$ is $-\tau^2 + \abs{\xi}_g^2$ which is clearly real valued, and one can check that the light curve $(\eta(t), t + \sigma)$ is indeed the projection of some null bicharacteristic curve $\gamma(t)$ as in (b). Below we will consider a general differential operator $P$ on an $n$-dimensional manifold $X$ and a curve $x: [0,T] \to X$ such that (a) and (b) hold.

\begin{remark}
We can give a microlocal explanation for the construction that follows. Since $p \in C^{\infty}(T^* X)$ is real, it induces a Hamilton vector field $H_p$ on $T^* X$. In local coordinates $(x,\xi)$ on $T^* X$ one has 
\[
H_p = (\nabla_{\xi} p(x,\xi), -\nabla_x p(x,\xi)).
\]
The integral curves of $H_p$ are called \emph{bicharacteristic curves} for $P$. Since $H_p p = 0$, the quantity $p$ is conserved along bicharacteristic curves. A bicharacteristic curve in $p^{-1}(0)$ is called a \emph{null bicharacteristic curve}. Given any $(x_0,\xi_0) \in p^{-1}(0)$, there is a unique bicharacteristic curve $\gamma(t)$ with $\gamma(0) = (x_0, \xi_0)$. In local coordinates one has $\gamma(t) = (x(t), \xi(t))$ where 
\begin{align*}
\dot{x}(t) &= \nabla_{\xi} p(x(t), \xi(t)), \\
\dot{\xi}(t) &= -\nabla_{x} p(x(t), \xi(t)).
\end{align*}

Differential operators (and even $\Psi$DOs) with real principal symbol satisfy the fundamental \emph{propagation of singularities} theorem: if $u$ solves $Pu = 0$, then $\mathrm{WF}(u)$ is contained in $p^{-1}(0)$ and it is invariant under the bicharacteristic flow. This is often used as a regularity result stating that if $u$ is smooth at $(x_0,\xi_0) \in p^{-1}(0)$ (i.e.\ $(x_0,\xi_0) \notin \mathrm{WF}(u)$), then $u$ is smooth on the whole null bicharacteristic through $(x_0,\xi_0)$. This result is sharp in the sense that if $\gamma$ is a suitable null bicharacteristic, there is an (approximate) solution of $Pu = 0$ whose wave front set is precisely on $\gamma$. Proposition \ref{prop_gelfand_concentrating_solutions} is a semiclassical example of such a result, and this yields interesting special solutions for $P$. It is this facet of propagation of singularities that is often useful for solving inverse problems. 
\end{remark}

\subsubsection{Deriving the eikonal equation}

Recall the geometric optics ansatz: we would like to construct approximate solutions to $Pv = 0$ in $X$ concentrating near the spatial projection $x(t)$ of a null bicharacteristic curve $\gamma(t)$, where $v$ has the form 
\[
v = e^{i\lambda \varphi} a.
\]
In the case of the wave equation, with $P = \p_t^2 - \Delta_g$, we computed in \eqref{wave_eq_conjugation_first} that 
\[
(\Box+q)v = e^{i\lambda \varphi} \big( \lambda^2 \left[ \abs{\nabla_x \varphi}_g^2 - (\p_t \varphi)^2 \right] a + i \lambda La + (\Box + q)a \big)
\]
where $L$ is a certain first order differential operator. Note that the quantity on brackets can be written as $p(x,t,\nabla_{x,t} \varphi(x,t))$ where $p(x,t,\xi,\tau) = -\tau^2 + \abs{\xi}_g^2$ the principal symbol of $P$. A similar direct computation can be performed for any differential operator $P$ of order $m$, and it gives  
\begin{equation} \label{p_conjugated}
Pv = e^{i\lambda \varphi} ( \lambda^m \left[ p(x, \nabla \varphi(x)) \right] a + i \lambda^{m-1} L a + O(\lambda^{m-2}) )
\end{equation}
where $L$ is a first order differential operator in $X$ given in local coordinates by 
\begin{equation} \label{l_explicit}
L = -\p_{\xi_j} p(x, \nabla \varphi(x)) \p_{x_j} + b(x)
\end{equation}
for some function $b$. In this argument we are always restricting to a local coordinate patch; in general one can glue together the functions from different coordinate patches.

Thus in order to kill the $\lambda^m$ term, ideally the phase function should satisfy the eikonal equation 
\[
p(x, \nabla \varphi(x)) = 0 \text{ in $X$}.
\]
Now instead of finding a smooth solution $\varphi$ in $X$, we wish to find a smooth function $\varphi$ near the curve $x(t)$ that solves the eikonal equation to infinite order on the curve, in the sense that 
\begin{equation} \label{eikonal_infinite_order}
\p^{\alpha}(p(x, \nabla \varphi(x)))|_{x(t)} = 0 \text{ for $t \in [0,T]$, for any multi-index $\alpha$.}
\end{equation}
It is clear that \eqref{eikonal_infinite_order} only involves derivatives of $\varphi$ on the curve $x(t)$. Thus it is our task to prescribe all these derivatives, i.e.\ the \emph{formal Taylor series} of $\varphi$ on the curve $x(t)$, so that \eqref{eikonal_infinite_order} holds. We can then use Borel summation to find a smooth function $\varphi$ near the curve $x(t)$ which has all the correct derivatives and thus satisfies \eqref{eikonal_infinite_order}.

\subsubsection{Solving the eikonal equation to infinite order on the curve}

Recall that $\gamma: [0,T] \to T^* X$ is a segment of a null bicharacteristic curve for $P$. In local coordinates $\gamma(t) = (x(t), \xi(t))$ one has the Hamilton equations 
\begin{align*}
\dot{x}(t) &= \nabla_{\xi} p(x(t), \xi(t)), \\
\dot{\xi}(t) &= -\nabla_{x} p(x(t), \xi(t)).
\end{align*}
Note that \eqref{eikonal_infinite_order} only involves $\nabla \varphi$ and its derivatives on $x(t)$. For the sake of definiteness we can choose 
\begin{equation} \label{varphi_zero}
\varphi(x(t)) := 0.
\end{equation}
Looking at zeroth order derivatives in \eqref{eikonal_infinite_order}, we see that $\varphi$ should satisfy 
\[
p(x(t), \nabla \varphi(x(t))) = 0.
\]
Now since $(x(t), \xi(t))$ is a \emph{null} bicharacteristic, this will hold if we choose 
\begin{equation} \label{nablavarphi_bicharacteristic}
\nabla \varphi(x(t)) := \xi(t).
\end{equation}
Looking at first order derivatives in \eqref{eikonal_infinite_order}, we get 
\[
\p_{x_j}(p(x, \nabla \varphi(x)))|_{x(t)} = \p_{x_j} p(x, \nabla \varphi) + \p_{\xi_l} p(x, \nabla \varphi) \p_{x_j x_l} \varphi \Big|_{x(t)}.
\]
This quantity should vanish, and indeed it always does: by the Hamilton equations $\p_{x_j} p(x, \nabla \varphi)|_{x(t)} = -\dot{\xi}_j(t)$, and by differentiating \eqref{nablavarphi_bicharacteristic} one has $\p_{\xi_l} p(x, \nabla \varphi) \p_{x_j x_l} \varphi|_{x(t)} = \dot{\xi}_j(t)$ so the two terms cancel.

Let us proceed to second order derivatives in \eqref{eikonal_infinite_order}. This will be the most important part of the construction. We compute (writing $p$ instead of $p(x, \nabla \varphi(x))$ for brevity) 
\begin{multline} \label{eikonal_second_derivative}
\p_{x_j x_k} (p(x, \nabla \varphi(x)))|_{x(t)} = \p_{x_j x_k} p + \p_{x_j \xi_l} p \,\p_{x_k x_l} \varphi + \p_{x_k \xi_l} p \, \p_{x_j x_l} \varphi \\
 + \p_{\xi_l \xi_m} p \, \p_{x_j x_l} \varphi \, \p_{x_k x_m} \varphi+ \p_{\xi_l} p \, \p_{x_j x_k x_l} \varphi \Big|_{x(t)}. 
\end{multline}
We wish to choose the second derivatives of $\varphi$ on $x(t)$, written as the matrix 
\begin{equation} \label{varphi_hessian}
H(t) := (\p_{x_j x_k} \varphi(x(t)))_{j,k=1}^n,
\end{equation}
so that the last expression vanishes. Since $\p_{\xi_l} p|_{x(t)} = \dot{x}_l(t)$, the last term in \eqref{eikonal_second_derivative} becomes 
\[
\p_{\xi_l} p \, \p_{x_j x_k x_l} \varphi |_{x(t)} = \frac{d}{dt} [\p_{x_j x_k} \varphi(x(t))] = \dot{H}_{jk}(t).
\]
Thus \eqref{eikonal_second_derivative} vanishes provided that $H(t)$ satisfies 
\begin{equation} \label{matrix_riccati}
\dot{H}+ BH + HB^t + HCH + F = 0
\end{equation}
where $B(t) := (\p_{x_j \xi_k} p)|_{x(t)}$, $C(t) := (\p_{\xi_j \xi_k} p)|_{x(t)}$, and $F(t) := (\p_{x_j x_k} p)|_{x(t)}$. The equation \eqref{matrix_riccati} is a \emph{matrix Riccati equation} for $H(t)$, i.e.\ a nonlinear ODE for $H(t)$ with quadratic nonlinearity. We will see below that one can always find a smooth solution in $[0,T]$ for the Riccati equation (this is where we need $\varphi$ to be complex, with $\mathrm{Im}(H(t))$ positive definite!).

We have now prescribed $\p^{\beta} \varphi(x(t))$ for $\abs{\beta} \leq 2$, so that \eqref{eikonal_infinite_order} holds for $\abs{\alpha} \leq 2$. Looking at third order derivatives in \eqref{eikonal_infinite_order} leads to a \emph{linear} ODE for the third order derivatives of $\varphi$ on the curve $x(t)$. Since linear ODEs with smooth coefficients always have smooth global solutions, we can prescribe the third order derivatives of $\varphi$. The same argument works for all the higher order derivatives. This concludes the construction of the formal Taylor series of $\varphi$ on the curve $x(t)$ so that \eqref{eikonal_infinite_order} holds. By Borel summation we obtain the required phase function $\varphi$.

\subsubsection{Solving the matrix Riccati equation}

The following result, which may also be found e.g.\ in \cite[Lemma 2.56]{KKL}, allows us to solve the Riccati equation \eqref{matrix_riccati}.

\begin{lemma}
Suppose that $B(t), C(t), F(t)$ are smooth real matrix functions on $[0,T]$ with $C$ and $F$ symmetric. Given any complex symmetric matrix $H_0$ with $\mathrm{Im}(H_0)$ positive definite, there is a unique smooth solution $H(t)$ of \eqref{matrix_riccati} in $[0,T]$ with $H(0) = H_0$. The matrix $H(t)$ is complex symmetric and $\mathrm{Im}(H(t))$ is positive definite on $[0,T]$.
\end{lemma}
\begin{proof}
(Sketch) For simplicity we only consider the equation 
\[
\dot{H}(t) + H(t)^2 = F(t), \qquad H(0) = H_0.
\]
The general case is analogous. Let us first consider the case $n=1$, i.e.\ $H(t)$ is a scalar function. We consider the ansatz 
\[
H(t) = \frac{z(t)}{y(t)}.
\]
Then 
\[
\dot{H}(t) + H(t)^2 = \frac{\dot{z} y - z \dot{y} + z^2}{y^2}.
\]
We note that this simplifies if $\dot{y} = z$. Thus we would like that $(z(t), y(t))$ solves the system 
\begin{equation} \label{z_y_system}
\left\{ \begin{array}{rl} \dot{z}(t) &\!\!\!= F(t) y(t), \\ \dot{y}(t) &\!\!\!= z(t). \end{array} \right.
\end{equation}
To get the initial conditions right, we would also like that 
\[
z(0) = H_0, \qquad y(0) = 1.
\]
Now this is a \emph{linear} system of ODEs, and hence always has a smooth solution $(z(t), y(t))$ on $[0,T]$.

The issue with this argument is that $y(t)$ might develop zeros, so that $H(t) = z(t)/y(t)$ would not be well defined. This can actually happen when the functions are \emph{real valued}. For instance, when $F(t) = 0$ and $H_0 = -1$, one has $z(t) = -1$ and $y(t) = 1-t$ so that $y(1) = 0$. Geometrically, if one is working with a geodesic, the zeros of $y(t)$ correspond to conjugate points. Hence on non-simple manifolds this argument would not give a smooth solution $H(t)$ on $[0,T]$.

However, if one lets $H(t)$ be \emph{complex} and chooses $H_0$ so that $\mathrm{Im}(H_0)$ is positive, this problem magically disappears. To see this, note that \eqref{z_y_system} gives (since $F$ is real) 
\[
\p_t(\bar{z} y - \bar{y} z) = F \abs{y}^2 + \abs{z}^2 - \abs{z}^2 - F \abs{y}^2 = 0.
\]
Now if $y(t_0) = 0$ for some $t_0$, the previous fact gives that 
\begin{align*}
0 &= \bar{z}(t_0) y(t_0) - \bar{y}(t_0) z(t_0) = \bar{z}(0) y(0) - \bar{y}(0) z(0) = \bar{H}_0 - H_0 \\
 &= -2i \mathrm{Im}(H_0).
\end{align*}
This is a contradiction. Hence $y(t)$ is nonvanishing on $[0,T]$ and $H(t)$ is well defined on $[0,T]$. The proof of the lemma in the general case can be concluded in a similar way by letting $z(t)$ and $y(t)$ be matrix valued.
\end{proof}

\subsubsection{Finding the amplitude $a$}

We have completed the construction of a smooth complex phase function $\varphi$ so that the eikonal equation holds to infinite order on the curve. To find the amplitude $a$ we go back to \eqref{p_conjugated} and write 
\[
a = a_0 + \lambda^{-1} a_{1} + \lambda^{-2} a_{2} + \ldots.
\]
Plugging this in \eqref{p_conjugated}, dropping the $p(x, \nabla \varphi(x))$ term (which is already very small), and looking at the highest power of $\lambda$ gives the equation 
\[
L a_0 = 0.
\]
This is a \emph{transport equation} (first order linear PDE). Again we are only interested in solving this equation to infinite order on the curve, i.e. 
\begin{equation} \label{transport_infinite_order}
\p^{\alpha}(La_0)|_{x(t)} = 0 \text{ for $t \in [0,T]$, for any multi-index $\alpha$.}
\end{equation}
Looking at zeroth order derivatives gives the equation 
\[
(L a_0)(x(t)) = 0.
\]
This is the most important part of the construction of $a$. But the explicit form for $L$ in \eqref{l_explicit}, and the fact that $\p_{\xi_j} p(x(t), \nabla \varphi(x(t))) = \dot{x}_j(t)$ by the Hamilton equations, show that this reduces to the equation 
\[
\p_t (a_0(x(t))) + c(t) a_0(x(t)) = 0.
\]
This is a \emph{linear} ODE which can always be solved (say with initial condition $a(x(0)) = 1$). Thus \eqref{transport_infinite_order} will hold for $\alpha=0$ and $a_0(x(t))$ will be a nonvanishing function.

Looking at higher order derivatives in \eqref{transport_infinite_order} gives similar linear ODE, and solving all of these and using Borel summation gives the desired function $a_0$. The functions $a_1, a_2, \ldots$ can be constructed in a similar way after taking into account some terms from $O(\lambda^{m-2})$ in \eqref{p_conjugated}. We can apply Borel summation to obtain the required amplitude $a$. Moreover, we can replace $a$ by $\chi a$ where $\chi$ a smooth cutoff function with $\chi = 1$ in a small neighborhood of $x([0,T])$, in order to have $a$ supported near the curve.

\subsubsection{Properties of $v$}

We have now constructed a function $v = e^{i\lambda \varphi} a$ where $\varphi$ and $a$ are smooth and the eikonal and transport equations are satisfied to infinite order on the curve $x(t)$. From \eqref{p_conjugated} we see that 
\[
Pv = e^{i\lambda \varphi}( \lambda^m q_m + \lambda^{m-1} q_{m-1} + \ldots)
\]
where each $q_j$ vanishes to infinite order on the curve $x(t)$. We wish to show that $Pv = O(\lambda^{-\infty})$ as $\lambda \to \infty$. We equip $X$ with some Riemannian metric and write $\delta(y)$ for the distance between $y$ and the curve $x(t)$. Since each $q_j$ vanishes to infinite order on the curve, for any $N$ there is $C_N > 0$ with 
\[
\abs{q_j(y)} \leq C_N \delta^N.
\]
Moreover, by \eqref{varphi_zero}, \eqref{nablavarphi_bicharacteristic}, \eqref{varphi_hessian} we have the Taylor series 
\[
\varphi(y) = \xi(t) \cdot (y-x(t)) + \frac{1}{2} H_{jk}(t) (y_j - x_j(t))(y_k-x_k(t)) + O(\delta^3).
\]
Here the term $\xi(t) \cdot (y-x(t))$ is real. Thus by the property $\mathrm{Im}(H(t)) > 0$ we have 
\begin{equation} \label{im_varphi_condition}
\mathrm{Im}(\varphi(y)) \geq c \delta(y)^2
\end{equation}
for some $c > 0$. This implies that for any $N$ there is $C_N$ with 
\[
\abs{e^{i\lambda \varphi} \lambda^j q_j} = e^{-\lambda \mathrm{Im}(\varphi)} \abs{\lambda^j q_j} \leq e^{-c \lambda \delta^2} \lambda^j  C_N \delta^N \leq C_N \lambda^{j-N/2}.
\]
Consequently, choosing $N$ large enough we obtain that for any $k$ there is $C_k > 0$ so that 
\[
\abs{Pv} \leq C_k \lambda^{-k} \text{ in $X$}
\]
as required.

Incidentally, \eqref{im_varphi_condition} explains why this is called a \emph{Gaussian beam} construction: the approximate solution has Gaussian decay away from the curve $x(t)$ since
\[
\abs{e^{i\lambda \varphi} a} = e^{-\lambda \mathrm{Im}(\varphi)} \abs{a} \leq C e^{-c \lambda \delta^2}.
\]
Moreover, the approximate solution is nontrivial since $a(x(t))$ is nonvanishing. After multiplying $v$ by a suitable constant, we have proved most of Proposition \ref{prop_gelfand_concentrating_solutions} (in the case of a general differential operator $P$ with real principal symbol). The remaining parts can be checked from the rather explicit form of $v$.

\begin{exercise}
Let $P = \p_t^2 - \Delta$ in $\mR^{n+1}$. Compute the null bicharacteristic curves of $P$.
\end{exercise}

\begin{exercise}
Let $P = \p_t^2 - \Delta$ in $\mR^{n+1}$, let $\omega \in S^{n-1}$, let $\lambda > 0$ and let $a \in C^{\infty}_c(\mR)$ have support equal to $\ol{B(0,\eps)}$. Show that the function 
\[
u(x,t) = e^{i\lambda(t-x \cdot \omega)} a(t-x \cdot \omega)
\]
solves $Pu = 0$ and satisfies $\mathrm{supp}(u) = \{ (x,t) \,:\, \abs{t-x \cdot \omega} \leq \eps \}$.
\end{exercise}

\begin{exercise}[Time-dependent case] \label{ex_wave}
Let $(M,g)$ be simple and assume that $q \in C^{\infty}_c(M^{\mathrm{int}} \times \mR)$. Consider the Dirichlet problem 
\begin{equation} \label{wavedp_r}
\left\{ \begin{array}{rll}
(\Box + q) u &\!\!\!= 0 & \quad \text{in } M \times \mR, \\
u &\!\!\!= f & \quad \text{on } \partial M \times \mR, \\
u &\!\!\!= 0 & \quad \text{for } t \ll 0.
\end{array} \right.
\end{equation}
Here $t \ll 0$ means that $t \leq -T_0$ for some $T_0 \geq 0$. You may assume that this problem is well-posed and for any $f \in C^{\infty}_c(\p M \times \mR)$ there is a unique solution $u \in C^{\infty}(M \times \mR)$. Consider the hyperbolic DN map 
\[
\Lambda_q: C^{\infty}_c(\p M \times \mR) \to C^{\infty}(\p M \times \mR), \ \ f \mapsto \p_{\nu} u|_{\p M \times \mR}.
\]
\begin{enumerate}
\item[(a)] 
Formulate a counterpart of Lemma \ref{lemma_gelfand_integral_identity} in this case.
\item[(b)] 
Formulate a counterpart of Proposition \ref{prop_gelfand_concentrating_solutions}. Which parts of the proof need to be modified?
\item[(c)] 
Use parts (a) and (b) to show that if $\Lambda_{q_1} = \Lambda_{q_2}$, then 
\[
\int_0^{\ell} q_1(\gamma(t), t+\sigma) \,dt = \int_0^{\ell} q_2(\gamma(t), t+\sigma) \,dt
\]
for any maximal geodesic $\gamma: [0,\ell] \to M$ and any $\sigma \in \mR$.
\item[(d)] 
Use the Fourier transform in $\sigma$ and injectivity of the geodesic X-ray transform in $(M,g)$ to invert the light ray transform in part (c) and to prove that $q_1 = q_2$. (\emph{Hint.} Look at the derivatives of the Fourier transform at $0$.)
\end{enumerate}
\end{exercise}

\section{Calder\'on problem} \label{sec_calderon}

Electrical Impedance Tomography (EIT) is an imaging method with applications in seismic and medical imaging and nondestructive testing. The method is based on the following important inverse problem.

\begin{quote}
{\bf Calder\'on problem:} Is it possible to determine the electrical conductivity of a medium by making voltage and current measurements on its boundary?
\end{quote}

In a standard formulation the medium is modelled by a bounded domain $\Omega \subset \mR^n$ with smooth boundary (in practice $n=3$), and the electrical conductivity is a positive function $\gamma \in C^{\infty}(\ol{\Omega})$. Under the assumption of no sources or sinks of current in $\Omega$, a voltage $f$ at the boundary induces a potential $u$ in $\Omega$ that solves the Dirichlet problem for the conductivity equation, 
\begin{equation} \label{conductivitydp}
\left\{ \begin{array}{rll}
\nabla \cdot \gamma \nabla u &\!\!\!= 0 & \quad \text{in } \Omega, \\
u &\!\!\!= f & \quad \text{on } \partial \Omega.
\end{array} \right.
\end{equation}
Since $\gamma \in C^{\infty}(\closure{\Omega})$ is positive, the equation is uniformly elliptic and there is a unique solution $u \in C^{\infty}(\ol{\Omega})$ for any boundary value $f \in C^{\infty}(\partial \Omega)$. One can define the (elliptic) Dirichlet-to-Neumann map (DN map) as 
\begin{equation*}
\Lambda_{\gamma}: C^{\infty}(\partial \Omega) \to C^{\infty}(\partial \Omega),  \ \ f \mapsto \gamma \p_{\nu} u |_{\partial \Omega}.
\end{equation*}
Here $\nu$ is the outer unit normal to $\p \Omega$ and $\p_{\nu} u|_{\p \Omega} = \nabla u \cdot \nu|_{\p \Omega}$ is the normal derivative of $u$. Physically, $\Lambda_{\gamma} f$ is the current flowing through the boundary.

The Calder\'on problem (also called the \emph{inverse conductivity problem}) is to determine the conductivity function $\gamma$ from the knowledge of the map $\Lambda_{\gamma}$. That is, if the measured current $\Lambda_{\gamma} f$ is known for all boundary voltages $f \in C^{\infty}(\p \Omega)$, one would like to determine the conductivity $\gamma$.

If the electrical properties of the medium depend on direction, which happens e.g.\ in muscle tissue, the medium is said to be \emph{anisotropic} and $\gamma = (\gamma^{jk})$ is a positive definite matrix function. When $n \geq 3$ one can write $\gamma^{jk} = \det(g)^{1/2} g^{jk}$ for some Riemannian metric $g$, and the conductivity equation becomes 
\[
\mathrm{div}_g(\nabla_g u) = 0.
\]
Thus Riemannian geometry appears already when considering anisotropic conductivities in Euclidean domains. More generally, if $(M,g)$ is a compact manifold with smooth boundary, we can consider the equation 
\begin{equation} \label{conductivity_riemannian}
\mathrm{div}_g(\gamma \nabla_g u) = 0
\end{equation}
for a positive function $\gamma \in C^{\infty}(M)$. This equation contains both equations above as a special case.

As a final reduction, if we replace $u$ by $\gamma^{-1/2} u$ in \eqref{conductivity_riemannian}, we obtain the equivalent \emph{Schr\"odinger equation} 
\[
(-\Delta_g + q)u = 0 \text{ in $M$}
\]
where $q = \frac{\Delta_g(\gamma^{1/2})}{\gamma^{1/2}}$. Assuming that $0$ is not a Dirichlet eigenvalue for $-\Delta_g + q$ in $M$, the DN map is given by 
\[
\Lambda_{q}: C^{\infty}(\partial M) \to C^{\infty}(\partial M),  \ \ f \mapsto \p_{\nu} u |_{\partial \Omega},
\]
where $f$ is the Dirichlet boundary value for the solution $u$. It is this equation that we will study in Section \ref{subseq_calderon_interior} when recovering the coefficients in the interior.

The Calder\'on problem is by now reasonably well understood in Euclidean domains \cite{SylvesterUhlmann, Nachman, AstalaPaivarinta, Bukhgeim}. Moreover, if $\dim(M) = 2$ and $M$ is simply connected, then isothermal coordinates, see Theorem \ref{thm_isothermal}, can be used to reduce the Riemannian case to the Euclidean case. We will thus assume from now on that $\dim(M) \geq 3$. In this case the problem is open in general, but there are results in special product geometries.

\begin{definition}
We say that $(M,g)$ is \emph{transversally anisotropic} if 
\[
(M, g) \subset \subset (\mR \times M_0, g), \qquad g = e \oplus g_0,
\]
where $(\mR, e)$ is the Euclidean line and $(M_0,g_0)$ is a compact $(n-1)$-manifold with boundary called the \emph{transversal manifold}.
\end{definition}

The definition means that $(M,g)$ is contained in some product manifold $\mR \times M_0$ with coordinates $(t, x)$ where $t \in \mR$ and $x \in M_0$, and the metric looks like 
\[
g(t, x) = \left( \begin{array}{cc} 1 & 0 \\ 0 & g_0(x) \end{array} \right).
\]
The Laplace-Beltrami operator has the form 
\[
-\Delta_g = -\p_t^2 - \Delta_x
\]
where $\Delta_x$ is the Laplace-Beltrami operator of $(M_0,g_0)$. Note that this looks similar to the Gelfand problem studied in Section \ref{sec_gelfand}, where we studied the wave operator $\p_t^2 - \Delta_x$. Formally the \emph{Wick rotation}, i.e.\ the map $t \mapsto it$, converts one equation to the other.

It turns out that, surprisingly, there are in fact analogies between the elliptic and hyperbolic inverse problems. One has the following counterpart of Theorem \ref{thm_gelfand_uniqueness} proved in \cite{DKSU, DKLS}.

\begin{theorem}[Uniqueness] \label{thm_calderon}
Let $(M,g)$ be a compact transversally anisotropic manifold. Assume also that the transversal manifold $(M_0,g_0)$ has injective geodesic X-ray transform. If $q_1, q_2 \in C^{\infty}(M)$ and if 
\[
\Lambda_{q_1} = \Lambda_{q_2},
\]
then $q_1 = q_2$ in $M$.
\end{theorem}

In particular, uniqueness holds by Theorem \ref{thm_xray} if the transversal manifold is \emph{simple}. By conformal invariance Theorem \ref{thm_calderon} holds more generally for metrics of the form $g = c(e \oplus g_0)$ for $c \in C^{\infty}(M)$ positive. The following questions remain open.

\begin{question} \label{question_transversal}
Is Theorem \ref{thm_calderon} true for any transversal manifold $(M_0,g_0)$?
\end{question}

\begin{question}
Is Theorem \ref{thm_calderon} true for any compact manifold $(M,g)$?
\end{question}

\begin{question}[Partial data] \label{question_partialdata}
If $\Omega \subset \mR^n$, $n \geq 3$, is a bounded domain and $\Gamma \subset \p \Omega$ is open, does the knowledge of $\Lambda_{\gamma} f|_{\Gamma}$ for all $f \in C^{\infty}_c(\Gamma)$ determine $\gamma$ uniquely?
\end{question}

Similarly as for the wave equation, it turns out that one can get better results for \emph{nonlinear} elliptic equations. Consider the model equation 
\begin{equation} \label{semilinear_dp}
\left\{ \begin{array}{rll}
-\Delta_g u + q u^3 &\!\!\!= 0 & \quad \text{in } M, \\
u &\!\!\!= f & \quad \text{on } \partial M.
\end{array} \right.
\end{equation}
In fact the method applies to the nonlinearities $q u^m$ for any integer $m \geq 3$. If $f \in C^{\infty}(\p M)$ is small (say in the $C^{2,\alpha}(\p M)$ norm), a Banach fixed point argument implies that \eqref{semilinear_dp} has a unique \emph{small} solution $u \in C^{\infty}(M)$. One can define the \emph{nonlinear DN map} 
\[
\Lambda_q^{\mathrm{NL}}: \{ f \in C^{\infty}(\p M) \;;\, \norm{f}_{C^{2,\alpha}(\p M)} < \delta \} \to C^{\infty}(\p M), \ \ \Lambda_q f = \p_{\nu} u|_{\p M}.
\]
It was proved independently in \cite{FO, LLLS} that Question \ref{question_transversal}, which is open for the linear Schr\"odinger equation, can be solved for the nonlinear equation \eqref{semilinear_dp}.

\begin{theorem}[Nonlinear case] \label{thm_calderon_semilinear}
Let $(M,g)$ be a compact transversally aniso\-tropic manifold, and let $q_1, q_2 \in C^{\infty}(M)$. If 
\[
\Lambda_{q_1}^{\mathrm{NL}} = \Lambda_{q_2}^{\mathrm{NL}},
\]
then $q_1 = q_2$ in $M$.
\end{theorem}

Question \ref{question_partialdata} has also been solved in the nonlinear case independently in \cite{KrupchykUhlmann, LLLS_partial}.

The rest of this section is organized as follows. In Section \ref{boundary_determination} we show that the Taylor series of the conductivity at the boundary is determined by the DN map. Interior determination is discussed in Section \ref{subseq_calderon_interior}, where we prove Theorem \ref{thm_calderon} by using complex geometrical optics solutions that are constructed in Section \ref{sec_cgo}. Finally, the case of nonlinear equations and the proof of Theorem \ref{thm_calderon_semilinear} is considered in Section \ref{subseq_calderon_nonlinear}.

\subsection{Boundary determination -- DN map as a $\Psi$DO} \label{boundary_determination}

We will first prove that the DN map $\Lambda_{\gamma}$ determines a scalar conductivity $\gamma \in C^{\infty}(\ol{\Omega})$ and all of its derivatives on $\p \Omega$. For simplicity we work in a Euclidean domain. The treatment in this section follows \cite{FSU}.

\begin{theorem}[Boundary determination] \label{thm_calderon_boundary_determination}
Let $\gamma_1, \gamma_2 \in C^{\infty}(\ol{\Omega})$ be positive. If 
\[
\Lambda_{\gamma_1} = \Lambda_{\gamma_2},
\]
then the Taylor series of $\gamma_1$ and $\gamma_2$ coincide at any point of $\p \Omega$.
\end{theorem}

This result was proved in \cite{KohnVogelius}, and it in particular implies that any real-analytic conductivity is uniquely determined by the DN map. The argument extends to piecewise real-analytic conductivities. A different proof was given in \cite{SylvesterUhlmann_boundary}, based on two facts:

\begin{enumerate}
\item[1.]
The DN map $\Lambda_{\gamma}$ is an elliptic $\Psi$DO of order $1$ on $\p \Omega$.
\item[2.] 
The Taylor series of $\gamma$ at a boundary point can be read off from the symbol of $\Lambda_{\gamma}$ computed in suitable coordinates. The symbol of $\Lambda_{\gamma}$ can be computed by testing against highly oscillatory boundary data (compare with \eqref{principal_symbol_oscillatory}).
\end{enumerate}

\begin{remark}
The above argument is based on studying the singularities of the integral kernel of the DN map, and it only determines the Taylor series of the conductivity at the boundary. The values of the conductivity in the interior are encoded in the $C^{\infty}$ part of the kernel, and different methods (based on \emph{complex geometrical optics solutions}) are required for interior determination as discussed in Section \ref{subseq_calderon_interior}.
\end{remark}

Let us start with a simple example.

\begin{example}[DN map in half space is a $\Psi$DO]
Let $\Omega = \mR^n_+ = \{ x_n > 0 \}$, so $\p \Omega = \mR^{n-1} = \{ x_n = 0 \}$. We wish to compute the DN map for the Laplace equation (i.e.\ $\gamma \equiv 1$) in $\Omega$. Consider 
\[
\left\{ \begin{array}{rll}
\Delta u &\!\!\!= 0 & \quad \text{in } \mR^n_+, \\
u &\!\!\!= f & \quad \text{on } \{ x_n = 0 \}.
\end{array} \right.
\]
Writing $x = (x',x_n)$ and taking Fourier transforms in $x'$ gives 
\[
\left\{ \begin{array}{rll}
(\p_n^2 - \abs{\xi'}^2) \hat{u}(\xi',x_n) &\!\!\!= 0 & \quad \text{in } \mR^n_+, \\
\hat{u}(\xi',0) &\!\!\!= \hat{f}(\xi'). &
\end{array} \right.
\]
Solving this ODE for fixed $\xi'$ and choosing the solution that decays for $x_n > 0$ gives 
\begin{align*}
 &\hat{u}(\xi',x_n) = e^{-x_n \abs{\xi'}} \hat{f}(\xi') \\
 &\implies u(x',x_n) = \mF_{\xi'}^{-1} \left\{ e^{-x_n \abs{\xi'}} \hat{f}(\xi') \right\}.
\end{align*}
We may now compute the DN map:
\[
\Lambda_1 f = -\p_n u|_{x_n = 0} =  \mF_{\xi'}^{-1} \left\{ \abs{\xi'} \hat{f}(\xi') \right\}.
\]
Thus the DN map on the boundary $\p \Omega = \mR^{n-1}$ is just $\Lambda_1 = \abs{D_{x'}}$ corresponding to the Fourier multiplier $\abs{\xi'}$. This shows that at least in this simple case, the DN map is an elliptic $\Psi$DO of order $1$.
\end{example}

We will now prove Theorem \ref{thm_calderon_boundary_determination} by an argument that avoids showing that the DN map is a $\Psi$DO, but is rather based on directly testing the DN map against oscillatory boundary data. The first step is a basic integral identity (sometimes called Alessandrini identity) for the DN map.

\begin{lemma}[Integral identity] \label{lemma_calderon_integral_identity}
Let $\gamma_1, \gamma_2 \in C^{\infty}(\ol{\Omega})$. If $f_1, f_2 \in C^{\infty}(\p \Omega)$, then 
\[
((\Lambda_{\gamma_1} - \Lambda_{\gamma_2}) f_1, f_2)_{L^2(\p \Omega)} = \int_{\Omega} (\gamma_1 - \gamma_2) \nabla u_1 \cdot \nabla \bar{u}_2 \,dx
\]
where $u_j \in C^{\infty}(\ol{\Omega})$ solves $\mdiv(\gamma_j \nabla u_j) = 0$ in $\Omega$ with $u_j|_{\p \Omega} = f_j$.
\end{lemma}
\begin{proof}
We first observe that the DN map is symmetric: if $\gamma \in C^{\infty}(\ol{\Omega})$ is positive and if $u_f$ solves $\nabla \cdot (\gamma \nabla u_f) = 0$ in $\Omega$ with $u_f|_{\p \Omega} = f$, then an integration by parts shows that 
\begin{align*}
(\Lambda_{\gamma} f, g)_{L^2(\p \Omega)} &= \int_{\p \Omega} (\gamma \p_{\nu} u_f) \ol{u}_g \,dS = \int_{\Omega} \gamma \nabla u_f \cdot \nabla \ol{u}_g \,dx \\
 &= \int_{\p \Omega} u_f (\ol{\gamma \p_{\nu} u_g}) \,dS = (f, \Lambda_{\gamma} g)_{L^2(\p \Omega)}.
\end{align*}
Thus 
\begin{align*}
(\Lambda_{\gamma_1} f_1, f_2)_{L^2(\p \Omega)} &= \int_{\Omega} \gamma_1 \nabla u_1 \cdot \nabla \ol{u}_2 \,dx, \\
(\Lambda_{\gamma_2} f_1, f_2)_{L^2(\p \Omega)} &= (f_1, \Lambda_{\gamma_2} f_2)_{L^2(\p \Omega)} = \int_{\Omega} \gamma_2 \nabla u_1 \cdot \nabla \ol{u}_2 \,dx.
\end{align*}
The result follows by subtracting the above two identities.
\end{proof}

Next we show that if $x_0$ is a boundary point, there is an approximate solution of the conductivity equation that concentrates near $x_0$, has highly oscillatory boundary data, and decays exponentially in the interior. As a simple example, the solution of 
\[
\left\{ \begin{array}{rll}
\Delta u &\!\!\!= 0 & \quad \text{in } \mR^n_+, \\
u(x',0) &\!\!\!= e^{i\lambda x' \cdot \xi'} &
\end{array} \right.
\]
that decays for $x_n > 0$ is given by $u = e^{-\lambda x_n} e^{i\lambda x' \cdot \xi'}$, which concentrates near $\{ x_n = 0 \}$ and decays exponentially when $x_n > 0$ if $\lambda$ is large. Roughly, this means that the solution of a Laplace type equation with highly oscillatory boundary data concentrates near the boundary. Note also that in a region like $\{ x_n > \abs{x'}^2 \}$, the function $u$ is harmonic and concentrates near the origin.

The following proposition makes these statements precise. Notice the similarity with Proposition \ref{prop_gelfand_concentrating_solutions} concerning solutions for wave equations that focus near geodesics.

\begin{proposition} \label{prop_calderon_oscillating_solutions}
(Concentrating approximate solutions) Let $\gamma \in C^{\infty}(\ol{\Omega})$ be positive, let $x_0 \in \p \Omega$, let $\xi_0$ be a unit tangent vector to $\p \Omega$ at $x_0$, and let $\chi \in C^{\infty}_c(\p \Omega)$ be supported near $x_0$. Let also $N \geq 1$. For any $\lambda \geq 1$ there exists $v = v_{\lambda} \in C^{\infty}(\ol{\Omega})$ having the form 
\[
v = \lambda^{-1/2} e^{i\lambda \Phi} a
\]
such that 
\begin{gather*}
\nabla \Phi(x_0) = \xi_0 -i \nu(x_0), \\
\text{$a$ is supported near $x_0$ with $a|_{\p \Omega} = \chi$},
\end{gather*}
and as $\lambda \to \infty$ 
\[
\norm{v}_{H^1(\Omega)} \sim 1, \qquad \norm{\mdiv(\gamma \nabla v)}_{L^2(\Omega)} = O(\lambda^{-N}).
\]
Moreover, if $\tilde{\gamma} \in C^{\infty}(\ol{\Omega})$ is positive and $\tilde{v} = \tilde{v}_{\lambda}$ is the corresponding approximate solution constructed for $\tilde{\gamma}$, then for any $f \in C(\ol{\Omega})$ and $k \geq 0$ one has 
\begin{equation} \label{boundary_determination_limit}
\lim_{\lambda \to \infty} \lambda^k \int_{\Omega} \mathrm{dist}(x,\p \Omega)^k f \nabla v \cdot \ol{\nabla \tilde{v}} \,dx = c_k \int_{\p \Omega} f \abs{\chi}^2 \,dS.
\end{equation}
for some $c_k \neq 0$.
\end{proposition}

We can now give the proof of the boundary determination result.

\begin{proof}[Proof of Theorem \ref{thm_calderon_boundary_determination}]
Using the assumption that $\Lambda_{\gamma_1} = \Lambda_{\gamma_2}$ together with the integral identity in Lemma \ref{lemma_calderon_integral_identity}, we have that 
\begin{equation} \label{calderon_integral_identity_vanishing}
\int_{\Omega} (\gamma_1 - \gamma_2) \nabla u_1 \cdot \nabla \bar{u}_2 \,dx = 0
\end{equation}
whenever $u_j$ solves $\mdiv(\gamma_j \nabla u_j) = 0$ in $\Omega$.

Let $x_0 \in \p \Omega$, let $\xi_0$ be a unit tangent vector to $\p \Omega$ at $x_0$, and choose $\chi \in C^{\infty}_c(\p \Omega)$ supported near $x_0$. We use Proposition \ref{prop_calderon_oscillating_solutions} to construct functions 
\[
v_j = v_{j,\lambda} = \lambda^{-1/2} e^{i\lambda \Phi} a_j
\]
so that $\nabla \Phi(x_0) = \xi_0 -i \nu(x_0)$, $a_j|_{\p \Omega} = \chi$ and 
\begin{equation} \label{vj_estimates}
\norm{v_j}_{H^1(\Omega)} \sim 1, \qquad \norm{\mdiv(\gamma_j \nabla v_j)}_{L^2(\Omega)} = O(\lambda^{-N}).
\end{equation}
We obtain exact solutions $u_j$ of $\mdiv(\gamma_j \nabla u_j) = 0$ by setting 
\[
u_j := v_j + r_j,
\]
where the correction terms $r_j$ are the unique solutions of 
\[
\mdiv(\gamma_j \nabla r_j) = -\mdiv(\gamma_j \nabla v_j) \text{ in $\Omega$}, \qquad r_j|_{\p \Omega} = 0.
\]
By standard energy estimates \cite[Section 6.2]{Evans} and by \eqref{vj_estimates}, the solutions $r_j$ satisfy 
\begin{equation} \label{calderon_correction_term_estimate}
\norm{r_j}_{H^1(\Omega)} \lesssim \norm{\mdiv(\gamma_j \nabla v_j)}_{H^{-1}(\Omega)} = O(\lambda^{-N}).
\end{equation}

We now insert the solutions $u_j = v_j + r_j$ into \eqref{calderon_integral_identity_vanishing}. Using \eqref{calderon_correction_term_estimate} and \eqref{vj_estimates}, it follows that 
\begin{equation} \label{gamma_difference_approximate}
\int_{\Omega} (\gamma_1 - \gamma_2) \nabla v_1 \cdot \nabla \bar{v}_2 \,dx = O(\lambda^{-N})
\end{equation}
as $\lambda \to \infty$. Letting $\lambda \to \infty$, the formula \eqref{boundary_determination_limit} yields 
\[
\int_{\p \Omega} (\gamma_1 - \gamma_2) \abs{\chi}^2 \,dS = 0.
\]
By varying $\chi$ we obtain $\gamma_1(x_0) = \gamma_2(x_0)$.

We will prove by induction that 
\begin{equation} \label{pnuj_induction_claim}
\p_{\nu}^j \gamma_1|_{\p \Omega} = \p_{\nu}^j \gamma_2|_{\p \Omega} \text{ near $x_0$ for any $j \geq 0$.}
\end{equation}
The case $j=0$ was proved above (here we may vary $x_0$ slightly). We make the induction hypothesis that \eqref{pnuj_induction_claim} holds for $j \leq k-1$. Let $(x', x_n)$ be boundary normal coordinates so that $x_0$ corresponds to $0$, and $\p \Omega$ near $x_0$ corresponds to $\{ x_n = 0 \}$. The induction hypothesis states that 
\[
\p_n^j \gamma_1(x',0) = \p_n^j \gamma_2(x',0), \qquad j \leq k-1.
\]
Considering the Taylor expansion of $(\gamma_1-\gamma_2)(x',x_n)$ with respect to $x_n$ gives that 
\[
(\gamma_1-\gamma_2)(x',x_n) = x_n^k f(x',x_n) \text{ near $0$ in $\{ x_n \geq 0 \}$},
\]
for some smooth function $f$ with $f(x',0) = \frac{\p_n^k (\gamma_1-\gamma_2)(x',0)}{k!}$. Inserting this formula in \eqref{gamma_difference_approximate}, we obtain that 
\[
\lambda^k \int_{\Omega} x_n^k f \nabla v_1 \cdot \nabla \bar{v}_2 \,dx = O(\lambda^{k-N}).
\]
Now $x_n = \mathrm{dist}(x,\p \Omega)$ in boundary normal coordinates. Assuming that $N$ was chosen larger than $k$, we may take the limit as $\lambda \to \infty$ and use \eqref{boundary_determination_limit} to obtain that 
\[
\int_{\p \Omega} f(x',0) \abs{\chi(x',0)}^2 \,dS(x') = 0.
\]
By varying $\chi$ we see that $\p_n^k (\gamma_1-\gamma_2)(x',0) = 0$ for $x'$ near $0$, which concludes the induction.
\end{proof}

It remains to prove Proposition \ref{prop_calderon_oscillating_solutions}, which constructs approximate solutions (also called \emph{quasimodes}) concentrating near a boundary point. This is a typical geometrical optics / WKB type construction for quasimodes with complex phase. The proof is elementary, although a bit long. The argument is simplified slightly by using the Borel summation lemma, which is used frequently in microlocal analysis in various different forms.

\begin{lemma}[Borel summation, {{\cite[Theorem 1.2.6]{Hormander}}}] \label{lemma_borel_summation}
Let $f_j \in C^{\infty}_c(\mR^{n-1})$ for $j = 0, 1, 2, \ldots$. There exists $f \in C^{\infty}_c(\mR^n)$ such that 
\[
\p_n^j f(x',0) = f_j(x'), \qquad j=0,1,2,\ldots.
\]
\end{lemma}

\begin{proof}[Proof of Proposition \ref{prop_calderon_oscillating_solutions}]
We will first carry out the proof in the case where $x_0 = 0$ and $\p \Omega$ is flat near $0$, i.e.\ $\Omega \cap B(0,r) = \{ x_n > 0 \} \cap B(0,r)$ for some $r > 0$ (the general case will be considered in the end of the proof). We also assume $\xi_0 = (\xi_0',0)$ where $\abs{\xi_0'} = 1$.

We look for $v$ in the form 
\[
v = e^{i\lambda \Phi} b.
\]
Write $Pu = D \cdot (\gamma D u) = \gamma D^2 u + D\gamma \cdot Du$ where $D = \frac{1}{i} \nabla$. The principal symbol of $P$ is 
\begin{equation} \label{ptwo_formula_simple}
p_2(x,\xi) := \gamma(x) \xi \cdot \xi.
\end{equation}
Since $e^{-i\lambda \Phi}D_j(e^{i\lambda \Phi} b) = (D_j + \lambda \p_j \Phi)b$, we compute 
\begin{align}
P(e^{i\lambda \Phi} b) &= e^{i\lambda \Phi}(D + \lambda \nabla \Phi) \cdot (\gamma (D + \lambda \nabla \Phi) b) \notag \\
 &= e^{i\lambda \Phi}\left[ \lambda^2 p_2(x,\nabla \Phi) b + \lambda \frac{1}{i} \left[ \underbrace{2 \gamma \nabla \Phi \cdot \nabla b + \nabla \cdot (\gamma \nabla \Phi) b}_{=: Lb} \right] + Pb \right]. \label{p_eilambdaphi_formula}
\end{align}

We want to choose $\Phi$ and $b$ so that $P(e^{i\lambda \Phi} b) = O_{L^2(\Omega)}(\lambda^{-N})$. Looking at the $\lambda^2$ term in \eqref{p_eilambdaphi_formula}, we first wish to choose $\Phi$ so that 
\begin{equation} \label{ptwo_phi_equation}
p_2(x,\nabla \Phi) = 0 \text{ in $\Omega$}.
\end{equation}
We additionally want that $\Phi(x',0) = x' \cdot \xi_0'$ and $\p_n \Phi(x',0) = i$ (this will imply that $\nabla \Phi(0) = \xi_0 + i e_n$). In fact, using \eqref{ptwo_formula_simple} we can just choose 
\[
\Phi(x',x_n) := x' \cdot \xi_0' + i x_n.
\]
Then $p_2(x,\nabla \Phi) = \gamma(\xi_0+ie_n) \cdot (\xi_0 + ie_n) \equiv 0$ in $\Omega$.

We next look for $b$ in the form 
\[
b = \sum_{j=0}^N \lambda^{-j} b_{-j}.
\]
Since $p_2(x,\nabla \Phi) \equiv 0$, \eqref{p_eilambdaphi_formula} implies that 
\begin{align}
P(e^{i\lambda \Phi} b) &= e^{i\lambda \Phi} \Big[ \lambda [ \frac{1}{i} Lb_0 ] + [ \frac{1}{i} Lb_{-1} + Pb_0] + \lambda^{-1} [ \frac{1}{i} Lb_{-2} + Pb_{-1}] + \ldots \notag \\
 &\hspace{30pt} + \lambda^{-(N-1)} [ \frac{1}{i} Lb_{-N} + Pb_{-(N-1)}] + \lambda^{-N} Pb_{-N }\Big] \label{p_eilambdaphi_formulatwo}.
\end{align}
We will choose the functions $b_{-j}$ so that 
\begin{align} \label{lbone} 
\left\{ \begin{array}{rl}
Lb_0 &\!\!\!= 0 \text{ to infinite order at $\{ x_n = 0 \}$}, \\[5pt]
\frac{1}{i} Lb_{-1} + P b_0 &\!\!\!= 0 \text{ to infinite order at $\{ x_n = 0 \}$}, \\[5pt]
&\!\vdots \\[5pt]
\frac{1}{i} Lb_{-N} + P b_{-(N-1)} &\!\!\!= 0 \text{ to infinite order at $\{ x_n = 0 \}$}. 
\end{array} \right.
\end{align}
We will additionally arrange that 
\begin{align} \label{bzeroone}
\left\{ \begin{array}{rl}
b_0(x',0) &\!\!\!= \chi(x'), \\[5pt]
b_{-j}(x',0) &\!\!\!= 0 \text{ for $1 \leq j \leq N$},
\end{array} \right.
\end{align}
and that each $b_{-j}$ is compactly supported so that 
\begin{equation} \label{bsupport}
\mathrm{supp}(b_{-j}) \subset Q_{\eps} := \{ \abs{x'} < \eps, \ 0 \leq x_n < \eps \}
\end{equation}
for some fixed $\eps > 0$.

To find $b_0$, we prescribe $b_0(x',0), \p_n b_0(x',0), \p_n^2 b_0(x',0)$, $\ldots$ successively and use the Borel summation lemma to construct $b_0$ with this Taylor series at $\{ x_n = 0 \}$. We first set $b_0(x',0) = \chi(x')$. Writing $\eta := \nabla \cdot (\gamma \nabla \Phi)$, we observe that 
\[
Lb_0|_{x_n=0} = 2 \gamma (\xi_0' \cdot \nabla_{x'} b_0 + i\p_n b_0) + \eta b_0|_{x_n=0}. 
\]
Thus, in order to have $Lb_0|_{x_n=0} = 0$ we must have 
\[
\p_n b_0(x',0) = -\frac{1}{2i \gamma(x',0)} \left[ 2 \gamma(x',0) \xi_0' \cdot \nabla_{x'} b_0 + \eta b_0 \right] \Big|_{x_n=0}.
\]
We prescribe $\p_n b_0(x',0)$ to have the above value (which depends on the already prescribed quantity $b_0(x',0)$). Next we compute 
\[
\p_n (Lb_0)|_{x_n=0} = 2 \gamma i \p_n^2 b_0 + Q(x', b_0(x',0), \p_n b_0(x',0))
\]
where $Q$ depends on the already prescribed quantities $b_0(x',0)$ and $\p_n b_0(x',0)$. We thus set 
\[
\p_n^2 b_0(x',0) = -\frac{1}{2i\gamma(x',0)} Q(x', b_0(x',0), \p_n b_0(x',0)),
\]
which ensures that $\p_n (Lb_0)|_{x_n=0} = 0$. Continuing in this way and using Borel summation, we obtain a function $b_0$ so that $Lb_0 = 0$ to infinite order at $\{ x_n = 0 \}$. The other equations in \eqref{lbone} are solved in a similar way, which gives the required functions $b_{-1}, \ldots, b_{-N}$. In the construction, we may arrange so that \eqref{bzeroone} and \eqref{bsupport} are valid.

If $\Phi$ and $b_{-j}$ are chosen in the above way, then \eqref{p_eilambdaphi_formulatwo} implies that 
\[
P(e^{i\lambda \Phi} b) = e^{i\lambda \Phi}\left[ \lambda q_1(x)   + \sum_{j=0}^{N} \lambda^{-j} q_{-j}(x) + \lambda^{-N} Pb_{-N} \right]
\]
where each $q_j(x)$ vanishes to infinite order at $x_n = 0$ and is compactly supported in  $Q_{\eps}$. Thus, for any $k \geq 0$ there is $C_k > 0$ so that $\abs{q_j} \leq C_k x_n^k$ in $Q_{\eps}$, and consequently 
\[
\abs{P(e^{i\lambda \Phi} b)} \leq e^{-\lambda \im(\Phi)} \left[ \lambda C_{k,N} x_n^k + C_N \lambda^{-N} \right].
\]
Since $\im(\Phi) = x_n$ in $Q_{\eps}$ we have 
\begin{align*}
\norm{P(e^{i\lambda \Phi} b)}_{L^2(\Omega)}^2 &\leq C_{k,N} \int_{Q_{\eps}} e^{-2\lambda x_n} \left[ \lambda^2 x_n^{2k} + \lambda^{-2N} \right] \,dx \\
 &\leq C_{k,N} \int_{\abs{x'} < \eps} \int_0^{\infty} e^{-2 x_n} \left[ \lambda^{1-2k} x_n^{2k} + \lambda^{-1-2N} \right] \,dx_n \,dx'.
\end{align*}
Choosing $k=N+1$ and computing the integrals over $x_n$, we get that 
\[
\norm{P(e^{i\lambda \Phi} b)}_{L^2(\Omega)}^2 \leq C_N \lambda^{-2N-1}.
\]
It is also easy to compute that 
\[
\norm{e^{i\lambda \Phi} b}_{H^1(\Omega)} \sim \lambda^{1/2}.
\]
Thus, choosing $a = \lambda^{-1/2} b$, we have proved all the claims except \eqref{boundary_determination_limit}.

To show \eqref{boundary_determination_limit}, we observe that 
\[
\nabla v = e^{i\lambda \Phi} \left[ i\lambda (\nabla \Phi) a + \nabla a \right].
\]
Using a similar formula for $\tilde{v} = e^{i\lambda \Phi} \tilde{a}$ (where $\Phi$ is independent of the conductivity), we have 
\[
\mathrm{dist}(x,\p \Omega)^k f \nabla v \cdot \ol{\nabla \tilde{v}} = x_n^k f e^{-2\lambda x_n} \left[ \lambda^2 \abs{\nabla \Phi}^2 a \ol{\tilde{a}} + \lambda^1 [ \cdots ] + \lambda^0 [ \cdots ]  \right].
\]
Now $\abs{\nabla \Phi}^2 = 2$ and $a = \lambda^{-1/2} b$ where $\abs{b} \lesssim 1$, and similarly for $\tilde{a}$. Hence 
\begin{align*}
 &\lambda^k \int_{\Omega} \mathrm{dist}(x,\p \Omega)^k f \nabla v \cdot \ol{\nabla \tilde{v}} \,dx \\
 &= \lambda^{k+1} \int_{\mR^{n-1}} \int_0^{\infty} x_n^k e^{-2\lambda x_n} f \left[ 2 b \ol{\tilde{b}} + O_{L^{\infty}}(\lambda^{-1}) \right] \,dx_n \,dx'.
\end{align*}
We can change variables $x_n \to x_n/\lambda$ and use dominated convergence to take the limit as $\lambda \to \infty$. The limit is 
\[
c_k \int_{\mR^{n-1}} f(x',0) b_0(x',0) \ol{\tilde{b}_0(x',0)} \,dx' = c_k \int_{\mR^{n-1}} f(x',0) \abs{\chi(x')}^2 \,dx'
\]
where $c_k = 2 \int_0^{\infty} x_n^k e^{-2x_n} \,dx_n \neq 0$.

The proof is complete in the case when $x_0 = 0$ and $\p \Omega$ is flat near $0$. In the general case, we choose boundary normal coordinates $(x',x_n)$ so that $x_0$ corresponds to $0$ and $\Omega$ near $x_0$ locally corresponds to $\{ x_n > 0 \}$. The equation $\nabla \cdot (\gamma \nabla u) = 0$ in the new coordinates becomes an equation 
\[
\nabla \cdot (\gamma A \nabla u) = 0 \text{ in $\{ x_n > 0 \}$}
\]
where $A$ is a smooth positive matrix only depending on the geometry of $\Omega$ near $x_0$. The construction of $v$ now proceeds in a similar way as above, except that the equation \eqref{ptwo_phi_equation} for the phase function $\Phi$ can only be solved to infinite order on $\{ x_n = 0 \}$ instead of solving it globally in $\Omega$.
\end{proof}

\subsection{Interior determination} \label{subseq_calderon_interior}

Assume that $(M,g)$ is a compact manifold with smooth boundary, and let $q \in C^{\infty}(M)$ be a potential. Consider the Dirichlet problem 
\begin{equation} \label{conductivity_dp}
\left\{ \begin{array}{rll}
(-\Delta_g + q) u &\!\!\!= 0 & \quad \text{in } M, \\
u &\!\!\!= f & \quad \text{on } \partial M.
\end{array} \right.
\end{equation}
We assume that $0$ is not a Dirichlet eigenvalue. Then for any $f \in C^{\infty}(\p M)$ there is a unique solution $u \in C^{\infty}(M)$. The boundary measurements are given by the (elliptic) DN map 
\[
\Lambda_q: C^{\infty}(\p M) \to C^{\infty}(\p M), \ \ \Lambda_q f = \p_{\nu} u|_{\p M}.
\]
The Calder\'on problem in this setting is to determine the potential $q$ from the knowledge of the DN map $\Lambda_q$, when the metric $g$ is known.

Let us now sketch the proof of Theorem \ref{thm_calderon}. The general scheme will be exactly the same as in the proof of Theorem \ref{thm_gelfand_xray} in the wave equation case, but with a few important differences. The proof proceeds in four steps:

\begin{enumerate}
\item[1.] 
Derivation of an integral identity showing that if $\Lambda_{q_1}=\Lambda_{q_2}$, then $q_1-q_2$ is $L^2$-orthogonal to certain products of solutions.
\item[2.] 
Construction of special solutions that concentrate near \emph{two-dimensional manifolds} $\mR \times \gamma$ where $\gamma$ is a maximal geodesic in $M_0$.
\item[3.] 
Inserting the special solutions in the integral identity and taking a limit, in order to recover integrals over geodesics.
\item[4.]
Inversion of the geodesic X-ray transform to prove that $q_1=q_2$.
\end{enumerate}


The first step, the integral identity, is completely analogous to the wave equation case.

\begin{lemma}[Integral identity] \label{lemma_calderon_integral_identity_riem}
Let $(M,g)$ be a compact manifold with boundary and let $q_1, q_2 \in C^{\infty}(M)$. If $f_1, f_2 \in C^{\infty}(\p M)$, then 
\[
((\Lambda_{q_1} - \Lambda_{q_2}) f_1, f_2)_{L^2(\p M)} = \int_{M} (q_1 - q_2)  u_1  \bar{u}_2 \,dV
\]
where $u_j \in C^{\infty}(M)$ solves $(-\Delta+q_j)u_j = 0$ in $M$ with $u_j|_{\p M} = f_j$.
\end{lemma}
\begin{proof}
We first observe that the DN map is symmetric: if $q \in C^{\infty}(M)$ is real valued and if $u_f$ solves $(-\Delta+q)u_f = 0$ in $M$ with $u_f|_{\p M} = f$, then an integration by parts shows that 
\begin{align*}
(\Lambda_{q} f, g)_{L^2(\p M)} &= \int_{\p M} (\p_{\nu} u_f) \ol{u}_g \,dS = \int_{M} (\langle \nabla u_f, \nabla \ol{u}_g \rangle + (\Delta u_f) \ol{u}_g) \,dV \\
 &= \int_{M} (\langle \nabla u_f, \nabla \ol{u}_g \rangle + q u_f \ol{u}_g) \,dV \\
 &= \int_{\p M} u_f \ol{ \p_{\nu} u_g} \,dS = (f, \Lambda_{q} g)_{L^2(\p M)}.
\end{align*}
Thus 
\begin{align*}
(\Lambda_{q_1} f_1, f_2)_{L^2(\p M)} &= \int_{M} (\langle \nabla u_1, \nabla \ol{u}_2 \rangle + q_1 u_1 \ol{u}_2) \,dV, \\
(\Lambda_{q_2} f_1, f_2)_{L^2(\p M)} &= (f_1, \Lambda_{q_2} f_2)_{L^2(\p M)} = \int_{M} (\langle \nabla u_1, \nabla \ol{u}_2 \rangle + q_2 u_1 \ol{u}_2) \,dV.
\end{align*}
The result follows by subtracting these two identities.
\end{proof}

By Lemma \ref{lemma_calderon_integral_identity_riem}, if $\Lambda_{q_1} = \Lambda_{q_2}$, then 
\begin{equation} \label{elliptic_orthogonality}
\int_{M} (q_1 - q_2)  u_1  \bar{u}_2 \,dV = 0
\end{equation}
for any solutions $u_j \in C^{\infty}(M)$ with $(-\Delta + q_j) u_j = 0$ in $M$.

\subsection{Special solutions -- complex geometrical optics} \label{sec_cgo}
We will next construct special solutions to the equation $(-\Delta + q) u = 0$. For simplicity we will do this assuming that the transversal manifold $(M_0,g_0)$ is \emph{simple}. This will make it possible to solve the eikonal equation globally, as in Section \ref{sec_wave_simple}. In the general case one can instead use a Gaussian beam construction as in Section \ref{subseq_gaussian_beam}.

Just like for the wave equation, we start with the geometric optics ansatz 
\begin{equation} \label{cgo_ansatz}
v(t,x) = e^{i\lambda \Phi(t,x)} a(t,x)
\end{equation}
where $\lambda \in \mR$ is a large parameter, $\Phi$ is a \emph{complex valued} phase function, and $a$ is an amplitude. The fact that the equation is elliptic requires us to use complex phase functions, and the corresponding solutions are called \emph{complex geometrical optics solutions}.

The construction of special solutions is similar to the wave equation case as in Proposition \ref{prop_gelfand_concentrating_solutions}. However, it has the following important differences which are consistent with the Wick rotation $t \mapsto it$:
\begin{itemize}
\item 
The phase function $\Phi$ solves the \emph{complex eikonal equation} 
\[
\langle \nabla_x \Phi, \nabla_x \Phi \rangle_{g_0} + (\p_t \Phi)^2 = 0,
\]
instead of $\abs{\nabla_x \varphi}_g^2 - (\p_t \varphi)^2 = 0$. The phase function $\Phi(x,t) = it - r$ is complex valued, instead of being real valued as in $\varphi(x,t) = t- r$.
\item 
The amplitude solves a \emph{complex transport equation} 
\[
2(\p_r + i \p_t)a + ba = 0,
\]
which has solutions concentrating near two-manifolds, instead of solving a real transport equation $2(\p_r + \p_t)a_0 + ba_0 = 0$ which has solutions concentrating near curves.
\item 
The solutions concentrate near two-dimensional manifolds $\mR \times \gamma$ where $\gamma$ is a maximal geodesic in $M_0$, instead of concentrating near curves $t \mapsto (\gamma(t), t+\sigma)$.
\item 
The approximate solutions $v = e^{i \lambda \Phi} a$ may grow exponentially in $\lambda$. Thus the exact solution $u = v + R$ cannot be constructed by solving a Dirichlet problem for $R$, but one must use a different solvability result (Carleman estimate).
\end{itemize}

\begin{remark}
Let us give a microlocal explanation of the argument above, following \cite{Salo_normal_forms}. As discussed in Remark \ref{rem_wave_sing},  propagation of singularities was responsible for the fact that one could construct special solutions for wave equations that concentrate near light rays. However, the equation $(-\Delta + q) u = 0$ in elliptic, so that singularities do not propagate (in fact any solution is $C^{\infty}$ by elliptic regularity) and one does not obtain special solutions in this way.

This can be remedied in the special case of transversally anisotropic manifolds, where $g = dt^2 + g_0$ has product form with coordinates $(t,x)$ and the Laplacian splits as $\Delta_g = \p_t^2 + \Delta_{x}$. We can introduce an (artificial) parameter $h = 1/\lambda$ and consider the exponentially conjugated \emph{semiclassical} operator 
\[
P := e^{t/h} h^2 (-\Delta_g + q) e^{-t/h}.
\]
This operator is not anymore elliptic when considered as a semiclassical operator. It is a so called semiclassical \emph{complex principal type} operator, and singularities for such operators propagate along two-dimensional surfaces in phase space called \emph{bicharacteristic leaves}. This implies that one can construct special solutions $v$ for $P$ concentrating along two-dimensional manifolds. One then obtains solutions $u = e^{-t/h} v$ to the elliptic equation $(-\Delta_g + q) u = 0$, and these are precisely the complex geometrical optics solutions mentioned above.
\end{remark}

We now discuss the argument in more detail. After applying the operator $-\Delta+q = -\p_t^2 - \Delta_x + q(t,x)$ to the ansatz \eqref{cgo_ansatz}, we obtain a direct analogue of the wave equation computation \eqref{wave_equation_geometric_optics}:
\begin{multline}
(-\p_t^2-\Delta_x+q)(e^{i\lambda \Phi} a) = e^{i\lambda \Phi} \big[ \lambda^2 \left[ \langle \nabla_x \Phi, \nabla_x \Phi \rangle_{g_0} + (\p_t \Phi)^2 \right] a \\
 - i \lambda \left[2 \p_t \Phi \p_t a + 2 \langle \nabla_x \Phi, \nabla_x a \rangle + (\Delta_{t,x} \Phi)a \right] + (-\Delta_{t,x} + q)a \big]. \label{elliptic_geometric_optics}
\end{multline}
Recall from \eqref{eikonal1} that in the wave equation case, the eikonal equation was 
\[
\abs{\nabla_x \varphi}_g^2 - (\p_t \varphi)^2 = 0
\]
and we used the solution 
\[
\varphi = t - r
\]
where $(r, \omega)$ were Riemannian polar coordinates in a neighborhood $U$ of the simple manifold $M_0$, with center outside $M_0$. Recall also that we were interested in solutions that concentrate near the geodesic $\gamma: r \mapsto (r, \omega_0)$ in $M_0$, where $\omega_0$ is fixed. In the elliptic case, the eikonal equation appearing in the $\lambda^2$ term in \eqref{elliptic_geometric_optics} is 
\[
\langle \nabla_x \Phi, \nabla_x \Phi \rangle_{g_0} + (\p_t \Phi)^2 = 0 \text{ in $M$}.
\]
We obtain a solution by choosing 
\[
\Phi(t,x) := it - r.
\]
This is consistent with the Wick rotation $t \mapsto it$.

Having solved the eikonal equation, \eqref{elliptic_geometric_optics} becomes 
\[
(-\Delta+q)(e^{i\lambda \Phi} a) = e^{i\lambda \Phi}(-i\lambda La + (-\Delta+q)a),
\]
where $L$ is the complex first order operator 
\[
La := 2 (\p_r + i \p_t) a + b a
\]
with $b := \Delta_{t,x} \Phi$. Here $\p_r + i \p_t$ is a Cauchy-Riemann, or $\dbar$, operator. We wish to find an amplitude solving 
\[
La = 0 \text{ in $M$}.
\]
Using coordinates $(t, r, \omega)$ where $(r,\omega)$ are polar coordinates as above, we choose the solution 
\[
a(t,r,\omega) = c(t,r,\omega)^{-1} \chi(\omega)
\]
where $c$ is an integrating factor solving $2 (\p_r + i \p_t) c = bc$ (this amounts to solving a $\dbar$ equation in $\mR^2$), and $\chi \in C^{\infty}_c(S^{n-2})$ is supported near $\omega_0$.

We have produced a function $v = e^{i\lambda \Phi} a$ satisfying 
\[
(-\Delta + q)v = e^{i\lambda \Phi}(-\Delta+q)a.
\]
Moreover, $a$ is supported near the two-dimensional manifold $\{ (t, r, \omega_0) \}$, which corresponds to the set $\mR \times \gamma$ where $\gamma$ is a geodesic in $M_0$. As in Section \ref{sec_gelfand} one could try to find an exact solution $u = v + R$ of $(-\Delta+q) u = 0$ in $M$ by solving the Dirichlet problem 
\begin{equation} \label{elliptic_correction_term}
\left\{ \begin{array}{rll}
(-\Delta + q)R &\!\!\!= -e^{i\lambda \Phi}(-\Delta+q)a & \quad \text{in } M, \\
R &\!\!\!= 0 & \quad \text{on } \partial M.
\end{array} \right.
\end{equation}
Now \emph{if} $\Phi$ were real valued, the right hand side would be $O_{L^2(M)}(1)$ as $\lambda \to \infty$ and at least one would get a correction term $R = O_{L^2(M)}(1)$. This could be converted to an estimate $R = O_{L^2(M)}(\lambda^{-1})$ by working with an amplitude $a = a_0 + \lambda^{-1} a_{-1}$ as in Section \ref{sec_gelfand}.

However, the phase function is \emph{not} real valued and in fact $e^{i\lambda \Phi} = e^{-\lambda t} e^{-i\lambda r}$. Thus the right hand side above is in general only $O(e^{C \lambda})$, which is not good since we wish to take the limit $\lambda \to \infty$. Instead of solving the Dirichlet problem, we need to use a different solvability result.

\begin{proposition}[Carleman estimate] \label{prop_carleman}
Let $(M,g)$ be transversally aniso\-tropic and let $q \in C^{\infty}(M)$. There are $C, \lambda_0 > 0$ so that whenever $\abs{\lambda} \geq \lambda_0$ and $f \in L^2(M)$, there is a function $R \in H^1(M)$ satisfying 
\[
(-\Delta+q)(e^{i \lambda \Phi} R) = e^{i\lambda \Phi} f \text{ in $M$}
\]
such that 
\[
\norm{R}_{L^2(M)} \leq \frac{C}{\abs{\lambda}} \norm{f}_{L^2(M)}.
\]
\end{proposition}
\begin{proof}
See e.g.\ \cite{DKSU}.
\end{proof}

We can now use Proposition \ref{prop_carleman} to convert the approximate solution $v = e^{i\lambda \Phi} a$ to an exact solution 
\[
u = e^{i\lambda \Phi} (a + R)
\]
of $(-\Delta+q)u=0$ in $M$, so that $\norm{R}_{L^2(M)} \to 0$ as $\abs{\lambda} \to \infty$. When $\abs{\lambda}$ is large, the solution $u$ is concentrated near the two-dimensional manifold 
\[
(\mR \times \gamma) \cap M
\]
but may grow exponentially in $\lambda$. However, the integral identity in Lemma \ref{lemma_calderon_integral_identity_riem} involves the product of two solutions, and we may take another solution of the type $e^{-i\lambda \Phi}(\tilde{a} + \tilde{R})$ so that the exponential growth will be cancelled in the product. By choosing such solutions and letting $\lambda \to \infty$ in \eqref{elliptic_orthogonality}, we obtain that the integral of $q_1-q_2$ (extended by zero outside $M$) over the two-dimensional manifold $\mR \times \gamma$ vanishes:
\[
\int_{-\infty}^{\infty} \int_0^{\ell} (q_1-q_2)(t, \gamma(r)) \,dr \,dt = 0.
\]
This is true for any maximal geodesic $\gamma$ in $(M_0,g_0)$, and hence using the injectivity of the geodesic X-ray transform on $(M_0,g_0)$ would give that 
\[
\int_{-\infty}^{\infty} (q_1-q_2)(t,x) \,dt = 0 \qquad \text{for all $x \in M_0$.}
\]
This is not quite enough to conclude that $q_1=q_2$. However, we can introduce an additional parameter $\sigma \in \mR$, which is analogous to the time delay parameter in the wave equation case. This can be done by performing the above construction with \emph{slightly complex frequency} $\lambda + i \sigma$. One obtains the following result analogous to Proposition \ref{prop_gelfand_concentrating_solutions}:

\begin{proposition}[Concentrating solutions] \label{prop_elliptic_concentrating}
Let $(M,g)$ be a transversally anisotropic manifold and let $q_1, q_2 \in C^{\infty}(M)$. Assume that the transversal manifold $(M_0,g_0)$ is simple, and that $\gamma: [0,\ell] \to M_0$ is a maximal geodesic. There is $\lambda_0 > 0$ so that whenever $\abs{\lambda} \geq \lambda_0$ and $\sigma \in \mR$, there are solutions $u_1 = u_{1,\lambda}$ of $(-\Delta+q_1)u_1 = 0$ in $M$ and $u_2 = u_{2,-\lambda}$ of $(-\Delta+q_2)u_2 = 0$ in $M$ such that for any $\psi \in C^{\infty}_c(M^{\mathrm{int}})$ one has 
\begin{equation} \label{orthogonality_calderon}
\lim_{\lambda \to \infty} \int_M \psi u_1 \ol{u}_2 \,dV = \int_{-\infty}^{\infty} \int_0^{\ell} e^{-i\sigma(t+ir)} \psi(t, \gamma(r)) \,dr \,dt.
\end{equation}
\end{proposition}

Theorem \ref{thm_calderon} now follows by inserting the solutions in Proposition \ref{prop_elliptic_concentrating} to the identity \eqref{elliptic_orthogonality}, taking the limit $\lambda \to \infty$, and using the Fourier transform in $t$ and injectivity of the geodesic X-ray transform in $(M_0,g_0)$ as in Exercise \ref{ex_wave}.

\subsection{Nonlinear equations} \label{subseq_calderon_nonlinear}

We will now consider the nonlinear equation 
\begin{equation} \label{semilinear_dp2}
\left\{ \begin{array}{rll}
-\Delta_g u + q u^3 &\!\!\!= 0 & \quad \text{in } M, \\
u &\!\!\!= f & \quad \text{on } \partial M
\end{array} \right.
\end{equation}
and the corresponding \emph{nonlinear DN map for small data} , 
\[
\Lambda_q^{\mathrm{NL}}: \{ f \in C^{\infty}(\p M) \;;\, \norm{f}_{C^{2,\alpha}(\p M)} < \delta \} \to C^{\infty}(\p M), \ \ \Lambda_q f = \p_{\nu} u|_{\p M}.
\]
We will prove Theorem \ref{thm_calderon_semilinear} which states that on transversally anisotropic manifolds, $\Lambda_{q_1}^{\mathrm{NL}} = \Lambda_{q_2}^{\mathrm{NL}}$ implies $q_1=q_2$.

A standard method for dealing with inverse problems for nonlinear equations is \emph{linearization}. Namely, if one knows the nonlinear DN map $\Lambda_q^{\mathrm{NL}}(f)$ for small $f$, then one also knows its linearization or Fr{\'e}chet derivative 
\[
(D\Lambda_q^{\mathrm{NL}})_0(h) = \p_{\eps} \Lambda_q^{\mathrm{NL}}(\eps h)|_{\eps=0}, \qquad h \in C^{\infty}(\p M).
\]
Let $u_{\eps}$ be the small solution of \eqref{semilinear_dp2} with boundary value $f = \eps h$, i.e.\ 
\begin{equation} \label{semilinear_dp3}
\left\{ \begin{array}{rll}
-\Delta_g u_{\eps} + q u_{\eps}^3 &\!\!\!= 0 & \quad \text{in } M, \\
u_{\eps} &\!\!\!= \eps h & \quad \text{on } \partial M.
\end{array} \right.
\end{equation}
Note that $u_0 = 0$, since $u = 0$ is the unique small solution with boundary value $0$. Formally differentiating \eqref{semilinear_dp3} in $\eps$ gives that 
\[
-\Delta_g (\p_{\eps} u_{\eps}) + 3 q u_{\eps}^2 \p_{\eps} u_{\eps} = 0.
\]
Setting $\eps = 0$ and using that $u_0 = 0$, we see that 
\[
v_h := \p_{\eps} u_{\eps}|_{\eps = 0}
\]
solves the linear equation 
\begin{equation} \label{linearized_dp}
\left\{ \begin{array}{rll}
-\Delta_g v_h &\!\!\!= 0 & \quad \text{in } M, \\
v_h &\!\!\!= h & \quad \text{on } \partial M.
\end{array} \right.
\end{equation}
Thus the linearized solution $v_h$ is just the harmonic function in $(M,g)$ with boundary value $h$. This formal computation can be justified. Since 
\[
(D\Lambda_q^{\mathrm{NL}})_0(h) = \p_{\eps} \Lambda_q^{\mathrm{NL}}(\eps h)|_{\eps=0} = \p_{\eps} \p_{\nu} u_{\eps}|_{\eps=0} = \p_{\nu} v_h,
\]
this leads to the following:

\begin{lemma}[Linearization of nonlinear DN map]
\[
(D\Lambda_q^{\mathrm{NL}})_0(h) = \Lambda_g h
\]
where $\Lambda_g$ is the DN map for the Laplace equation \eqref{linearized_dp}.
\end{lemma}

This shows that from the knowledge of $\Lambda_q^{\mathrm{NL}}$, we can recover its linearization $(D\Lambda_q^{\mathrm{NL}})_0 = \Lambda_g$. However, this \emph{first linearization} does not contain any information about the unknown potential $q$. It turns out that for the nonlinearity $q u^3$, the right thing to do is to look at the \emph{third linearization}, i.e.\ the \emph{third order Fr{\'e}chet derivative}, $(D^3 \Lambda_q^{\mathrm{NL}})_0$.

The third linearization can be computed by considering Dirichlet data of the form $f = \eps_1 h_1 + \eps_2 h_2 + \eps_3 h_3$ where $h_j \in C^{\infty}(\p M)$ and $\eps_j > 0$ are small. Writing $\eps = (\eps_1,\eps_2,\eps_3)$, let $u_{\eps}$ be the solution of 
\begin{equation} \label{semilinear_dp4}
\left\{ \begin{array}{rll}
-\Delta_g u_{\eps} + q u_{\eps}^3 &\!\!\!= 0 & \quad \text{in } M, \\
u_{\eps} &\!\!\!= \eps_1 h_1 + \eps_2 h_2 + \eps_3 h_3 & \quad \text{on } \partial M.
\end{array} \right.
\end{equation}
We formally apply the derivative $\p_{\eps_1 \eps_2 \eps_3}$ to this equation to obtain 
\begin{align*}
0&=-\Delta_g (\p_{\eps_1 \eps_2 \eps_3} u_{\eps}) + q \p_{\eps_1 \eps_2 \eps_3}(u_{\eps}^3) \\
 &= -\Delta_g (\p_{\eps_1 \eps_2 \eps_3} u_{\eps}) + q \p_{\eps_1 \eps_2}(3 u_{\eps}^2 \p_{\eps_3} u_{\eps}) \\
 &= -\Delta_g (\p_{\eps_1 \eps_2 \eps_3} u_{\eps}) + q \p_{\eps_1}(6 u_{\eps} \p_{\eps_2} u_{\eps} \p_{\eps_3} u_{\eps} + 3 u_{\eps}^2 \p_{\eps_2 \eps_3} u_{\eps}) \\
 &= -\Delta_g (\p_{\eps_1 \eps_2 \eps_3} u_{\eps}) + 6q \p_{\eps_1} u_{\eps} \p_{\eps_2} u_{\eps} \p_{\eps_3} u_{\eps} + \ldots
\end{align*}
where $\ldots$ consists of terms that contain a power of $u_{\eps}$. Since $u_0 = 0$, when we set $\eps = 0$ all the terms in $\ldots$ will vanish. Thus 
\[
w := \p_{\eps_1 \eps_2 \eps_3} u_{\eps}|_{\eps=0}
\]
will solve the equation (recall that $v_{h_j} = \p_{\eps_j} u_{\eps}|_{\eps=0}$) 
\begin{equation} \label{third_linearized_dp}
\left\{ \begin{array}{rll}
-\Delta_g w &\!\!\!= - 6 q v_{h_1} v_{h_2} v_{h_3} & \quad \text{in } M, \\
w &\!\!\!= 0 & \quad \text{on } \partial M.
\end{array} \right.
\end{equation}
Now if the know the nonlinear DN map $\Lambda_q^{\mathrm{\mathrm{NL}}} (\eps_1 h_1 + \eps_2 h_2 + \eps_3 h_3) = \p_{\nu} u_{\eps}$, then we also know $\p_{\nu} w = \p_{\nu}  \p_{\eps_1 \eps_2 \eps_3} u_{\eps}|_{\eps = 0}$. Thus for any $h_4 \in C^{\infty}(\p M)$, we also know 
\[
\int_{\p M} (\p_{\nu} w) h_4 \,dS = \int_M ((\Delta_g w) v_{h_4} + \langle \nabla w, \nabla v_{h_4} \rangle_g) \,dV.
\]
Integrating by parts in the last term and using $w|_{\p M} = 0$ and $\Delta_g v_{h_4} = 0$, we obtain that 
\[
\int_{\p M} (\p_{\nu} w) h_4 \,dS = 6 \int_M q v_{h_1} v_{h_2} v_{h_3} v_{h_4} \,dV.
\]
Since $\p_{\nu} w$ is determined by $\Lambda_q^{\mathrm{\mathrm{NL}}}$, also the right hand side is determined by the map $\Lambda_q^{\mathrm{\mathrm{NL}}}$. (In fact one can verify that the left hand side is equal to $((D^3 \Lambda_q^{\mathrm{NL}})_0(h_1,h_2,h_3), h_4)_{L^2(\p M)}$, where $(D^3 \Lambda_q^{\mathrm{NL}})_0$ is the third Fr{\'e}chet derivative of $\Lambda_q^{\mathrm{NL}}$ considered as a trilinear form.) This formal argument can be justified and it leads to the following identity.

\begin{lemma}[Integral identity in nonlinear case] \label{lemma_identity_nonlinear}
If $\Lambda_{q_1}^{\mathrm{\mathrm{NL}}} = \Lambda_{q_2}^{\mathrm{\mathrm{NL}}}$, then 
\[
\int_M (q_1-q_2) v_{1} v_{2} v_{3} v_{4} \,dV = 0
\]
for any $v_j \in C^{\infty}(M)$ solving $\Delta_g v_j = 0$ in $M$.
\end{lemma}

This integral identity related to the nonlinear equation $-\Delta_g u + q u^3 = 0$ has two benefits over the identity for the linear equation $-\Delta_g u + qu = 0$:
\begin{itemize}
\item 
$q_1-q_2$ is $L^2$-orthogonal to products of \emph{four} solutions, instead of products of two solutions;
\item 
the solutions $v_{j}$ are solutions of the Laplace equation $\Delta_g v_{j} = 0$, which does not contain the unknown potential $q$.
\end{itemize}

Let us finally sketch how one proves Theorem \ref{thm_calderon_semilinear} based on the integral identity in Lemma \ref{lemma_identity_nonlinear} and the construction of special solutions in Proposition \ref{prop_elliptic_concentrating}. The main point is that instead of considering a fixed geodesic $\gamma$ in $(M_0,g_0)$, one can consider \emph{two intersecting geodesics}.

Suppose that $\gamma_1$ and $\gamma_2$ are two maximal geodesics in $(M_0,g_0)$ that intersect only at one point $x_0 \in M_0$. We use Proposition \ref{prop_elliptic_concentrating} to find two harmonic functions $v_{\lambda}$ and $v_{-\lambda}$ in $M$ so that the product $v_{\lambda} \ol{v}_{-\lambda}$ is concentrated near $\mR \times \gamma_1$. We similarly choose two harmonic functions $w_{\lambda}$ and $w_{-\lambda}$ in $M$ so that the product $w_{\lambda} \ol{w}_{-\lambda}$ is concentrated near $\mR \times \gamma_2$. Then the product 
\[
v_{\lambda} \ol{v}_{-\lambda} w_{\lambda} \ol{w}_{-\lambda}
\]
is concentrated near the \emph{one-dimensional manifold} $\mR \times \{ x_0 \}$, and one has 
\[
0 = \lim_{\lambda \to \infty} \int_M (q_1-q_2) v_{\lambda} \ol{v}_{-\lambda} w_{\lambda} \ol{w}_{-\lambda} \,dV = \int_{-\infty}^{\infty} e^{-i \sigma t} (q_1-q_2)(t, x_0) \,dt.
\]
The point is that one has concentration at a \emph{single point} $x_0$ in $M_0$, instead of concentration near a fixed geodesic in $M_0$. It follows that the Fourier transform of $(q_1-q_2)(\,\cdot\,, x_0)$ vanishes identically for every $x_0 \in M_0$. This implies that $q_1=q_2$.

In general, given $x_0 \in M_0$ it may not be possible to find two finite length geodesics that only intersect at $x_0$. The possibility of multiple intersection points can be handled by introducing another extra parameter in the solutions, see \cite{LLLS} for details. This proves Theorem \ref{thm_calderon_semilinear} in general.

\begin{exercise}
Let $(M,g)$ be a compact Riemannian manifold with smooth boundary and let $\gamma \in C^{\infty}(M)$ be positive. Show that setting $v = \gamma^{-1/2} u$ reduces the conductivity equation $\mathrm{div}_g(\gamma \nabla_g v) = 0$ to the Schr\"odinger equation $(-\Delta_g + q)u = 0$ where $q = \frac{\Delta_g(\gamma^{1/2})}{\gamma^{1/2}}$.
\end{exercise}

\begin{exercise}
Compute the DN map for the Laplace equation in the unit disk $\mathbb{D} \subset \mR^2$ in terms of the Fourier expansion of the Dirichlet data $f$.
\end{exercise}

\begin{exercise}
Let $(M,g)$ be a simply connected two-dimensional manifold with boundary. Characterize all complex functions $\Phi \in C^{\infty}(M)$ that solve the complex eikonal equation $\langle \nabla_g \Phi, \nabla_g \Phi \rangle_g = 0$ in $M$, where $\langle \,\cdot\,, \,\cdot\, \rangle_g$ is understood as a complex bilinear (instead of sesquilinear) form. (Hint: Theorem \ref{thm_isothermal} may be helpful if you are not familiar with Riemannian geometry.)
\end{exercise}

\begin{exercise}
Let $(M,g)$ be transversally anisotropic with simple transversal manifold $(M_0,g_0)$. Let $\psi \in C^{\infty}_c(M^{\mathrm{int}})$, and assume that the right hand side of \eqref{orthogonality_calderon} vanishes for any $\sigma \in \mR$ and any maximal geodesic $\gamma$ in $(M_0,g_0)$. Prove that $\psi = 0$.
\end{exercise}

\addtocontents{toc}{\SkipTocEntry}

\end{document}